\documentclass[11pt]{amsart}

\usepackage{amssymb}
\usepackage{amsmath}
\usepackage{amsfonts}
\usepackage{graphicx}
\usepackage{graphics}
\usepackage{amsbsy}
\usepackage{amscd}

\usepackage{color}

\usepackage[all,knot,poly]{xy}
\usepackage{ifpdf}
\usepackage{ucs}
\usepackage[utf8x]{inputenc}
\usepackage[T1]{fontenc}
\usepackage{latexsym}
\usepackage{array}
\usepackage{enumerate}

\def\zet{\mathbb{Z}}

\def\zetita{\mathbb{Z}}

\def\TC{\overline{\mathrm{TC}}}
\def\P{\mathrm{P}}

\newtheorem{theorem}{Theorem}[section]
\newtheorem{proposition}[theorem]{Proposition}
\newtheorem{lema}[theorem]{Lemma}
\newtheorem{definition}[theorem]{Definition}
\newtheorem{corolario}[theorem]{Corollary}
\newtheorem{nota}[theorem]{Remark}
\newtheorem{conjetura}[theorem]{Conjecture}

\title[Topological complexity of lens spaces]{Biequivariant maps on spheres and topological complexity of lens spaces}

\author{Jes\'us Gonz\'alez, Maurilio Velasco, and W.~Stephen Wilson}
\thanks{The first author was supported by the Conacyt Grant 102783 during the time this research was conducted.}
\date{\today}

\begin{document}

\begin{abstract}
Weighted cup-length calculations in singular cohomology led Farber and Grant in 2008 to general lower bounds for the topological complexity of lens spaces. We replace singular cohomology by K-theory, and weighted cup-length arguments by considerations with biequivariant maps on spheres to  improve on Farber-Grant's bounds by arbitrarily large amounts. Our calculations are based on the identification of key elements conjectured to generate the annihilator ideal of the toral bottom class in the $ku$-homology of the classifying space for $\zetita_{2^k}\times\zetita_{2^e}$.
\end{abstract}

\maketitle
\tableofcontents

\section{Topological complexity of 2-torsion lens spaces}\label{TCresults}
The concept of topological complexity was
introduced by M.~Farber in~\cite{F} motivated
by one of the most basic problems in robotics: 
given a mechanical system~${\mathcal S}$, one wants to
determine an algorithm or a program capable of taking the system
from any given initial state $A$ to any given final state $B$ under certain
given constraints. Thus, a motion planning program for 
${\mathcal S}$ is a set of rules that
specify a movement of the system from any given initial state to
any other given final state. The problem can be formalized
mathematically in the following way. Let $X$ denote the
configuration space of ${\mathcal S}$, and let $PX$ be the function space of all
continuous paths $\gamma:[0,1] \rightarrow X$ in $X$. There is
a fibration $\pi_X:PX\rightarrow X\times X$ that associates to every
$\gamma\in PX$ the ordered pair formed by the initial and final 
points of $\gamma$, i.e. $\pi_X(\gamma)=(\gamma(0),\gamma (1))$. 
In its most basic form, the motion planning problem in $X$
asks to construct a function $s:X\times X\rightarrow PX$ such that
the composition $\pi_X\circ s$ is the identity; that is, $s$ must be a
cross-section of $\pi_X$. The natural restriction 
that the section $s$ be continuous can only hold when
$X$ is contractible; in general, continuity will hold on
neighborhoods $U_i$ covering $X\times X$, called {\em local domains}, 
on each of which $\pi_X$ should admit a continuous 
local cross-section $s_i$, called
the $i$-th {\em local rule}. Such a set of local domains and local rules
is called a {\em motion planner} in $X$. 

\medskip
The (normalized) topological complexity of $X$,
denoted here by $\TC(X)$, is one less than the lowest possible number of 
local rules among motion planners in $X$. In other words, $\TC(X)$ is 
the normalized Schwartz genus of $\pi_X$ (so, the standard
convention for the Lusternik–Schnirelmann category is now imposed on 
Farber's original definition).

\medskip
The concept of topological complexity captures a number of interesting
phenomena. To begin with, as shown by Farber, $\TC(X)$ gives a sharp
measure of the intrinsic discontinuities in the 
motion planning problem in $X$. On the other hand, $\TC(X)$ depends only 
on the homotopy type of $X$ and, 
since the diagonal $\Delta:X\to X\times X$ is homotopy equivalent 
to $\pi_X$, the weighted cup-length of the zero-divisors in any 
multiplicative cohomology theory of $X$ gives lower bounds for 
$\TC(X)$. The latter is the key observation leading (with singular cohomology) 
to Farber-Grant's general lower bounds for the topological complexity of
lens spaces (Theorem~\ref{fargra} in the appendix).
But most striking is the connection with a classical problem
in differential topology:
For the $n$-dimensional real projective space $\P^n$,~\cite{FTY} shows
\begin{eqnarray}\label{TCImm}
\TC(\P^n)=\mathrm{Imm}(\P^n)
\end{eqnarray}
when $n\neq 1,3,7$. Here $\mathrm{Imm}(\P^n)$ stands for the Euclidean 
immersion dimension of $\P^n$, i.e.~the smallest positive integer $d$ such that
$\P^n$ can be immersed in $\mathbb{R}^d$. 

\medskip
The above property can be extrapolated to the case 
of lens spaces by recalling from~\cite{AGJ} that, for $n\neq 1,3,7$,
$\mathrm{Imm}(\P^n)$ is the smallest positive integer $m$ for which 
there is a $\mathbb{Z}_2$-biequivariant map $S^n\times S^n\to S^m$, 
where $\mathbb{Z}_2$ acts antipodally on each sphere.

\begin{definition}\label{ese}
For positive integers $\hspace{.2mm}n$ and $\hspace{.2mm}t$ consider the standard $\hspace{.1mm}\mathbb{Z}_{t}$-action on the $(2n+1)$-dimensional sphere $S^{2n+1}$, and let $b_{n,t}$ denote the smallest positive integer $m$ for which there is a $\mathbb{Z}_{t}$-biequivariant map $S^{2n+1}\times S^{2n+1}\to S^{2m+1}$.
\end{definition}

Let $L^{2n+1}(t)$ denote the standard $(2n+1)$-dimensional $t$-torsion lens space, the orbit space of the action in Definition~\ref{ese} above.
The main result in~\cite{G} partially extends~(\ref{TCImm}) by showing
\begin{eqnarray}\label{valCT}
\TC(L^{2n+1}(t))=2\hspace{.2mm}b_{n,t}+\varepsilon_{n,t},
\quad\varepsilon_{n,t}\geq0, 
\end{eqnarray} 
where in fact $\varepsilon_{n,t}\in\{0,1\}$ if $t$ is even.\footnote{There is a gap in the proof given in~\cite{G} of the estimates for $\varepsilon_{n,t}$. The first author thanks Jos\'e Garc\'{\i}a-Calcines and Lucile Vandembroucq for noticing the problem and for discussions leading to a fixing of the gap. Details are discussed in the appendix of this paper.}

\medskip
We focus on $b_{n,2^e}$, which will also be denoted by $b(n,e)$. While $b(n,1)$ captures up to parity the (still undetermined) immersion dimension of $\P^{2n+1}$, the function $b(n,e)$ has an easy description for $e$ large enough:~\cite[Proposition~2.2]{G} claims
\begin{equation}\label{topcases}
b(n,e)=\begin{cases}2n,&e>\alpha(n);\\2n-1,&e=\alpha(n),\end{cases}
\end{equation}
where $\alpha (n)$ denotes the number of ones that appear
in the binary expansion of $n$. This paper's goal (Theorem~\ref{CTC} below) 
is to analyze the ``first'' unsolved case in~(\ref{topcases}): 
$e =\alpha(n)-1$. 

\medskip
The main result in~\cite{DD} claims that $P^{2(m+\alpha(m)-1)}$ does not admit an immersion in $\mathbb{R}^{4m-2\alpha(m)}$. Therefore (\ref{TCImm}) and~(\ref{valCT}) imply
\begin{eqnarray}\label{cota}
b(m+\alpha(m)-1,1)\geq2m-\alpha(m).
\end{eqnarray}
More generally, considerations on the immersion dimension of lens spaces led to the following conjecture in~\cite{GZ}:

\begin{conjetura}\label{conejura}
For $1\leq e\leq\alpha(m)$, $b(m+\alpha(m)-e,e)\geq 2m-\alpha(m)+e-1$.
\end{conjetura}

This should be considered as an alternative point of view toward an eventual understanding of the intricacies in the ``$\mathrm{TC}$-approach'' to the immersion dimension of {\em odd-dimensional} projective spaces (c.f.~\cite{G}). Namely, although half a century of experience suggests that the numeric value of~(\ref{TCImm}) might look like $2n-k(n)\alpha(n)+o(\alpha(n))$ with $1\leq k(n)\leq 6$, determining the nature of $k(n)$ and the form of $o(\alpha(n))$ is currently a major open task. However, if the ``correcting term'' $\alpha(m)-1$ in the first entry of the $b$-function in~(\ref{cota}) were to be ignored, $k(n)=2$ and $o(\alpha(n))\geq0$ would provide a rather satisfying general lower bound for~(\ref{TCImm}), at least in the case of odd dimensional projective spaces. Conjecture~\ref{conejura} extends such an idealistic bound to (the topological complexity of) $2^e$-torsion lens spaces in such a way that the required correcting term $\alpha(m)-e$ gets smaller as the lens-space torsion increases. For instance, after ignoring the correcting term ``${}+1$'', Theorem~\ref{CTC} below should be thought of as giving the {\em expected} answer for the next case in~(\ref{topcases}).

\medskip
In view of~(\ref{topcases}) and~(\ref{cota}), Conjecture~\ref{conejura}
is true for $e=1$, and sharp for $e=\alpha(m)$. It is also 
known to hold in many other cases with $e=2$ (cf.~\cite[Theorem~2.4]{GZ}).
One of the main achievements of this paper is a proof of 
Conjecture~\ref{conejura} for $e=\alpha(m)-1$, with its corresponding 
application to the topological complexity of lens spaces.

\begin{theorem}\label{CTC}
$b(m+1,\alpha(m)-1)\geq 2m-2$ provided $\alpha(m)\geq2$.
\end{theorem}

\begin{corolario}\label{cotainferior}
$\TC(L^{2m+3}(2^{\alpha(m)-1}))\ge 4m-4$ provided $\alpha(m)\geq2$.
\end{corolario}

The proof of Theorem~\ref{CTC} is given in Section~\ref{lasecnueva}. Corollary~\ref{cotainferior} improves by arbitrarily large amounts on Farber-Grant's general lower bound for the topological complexity of lens spaces. Indeed, if $\nu(m)$ stands for the exponent in the largest 2-power dividing $m$ then, as indicated in Proposition~\ref{Compara} in the appendix,~\cite[Theorem~11]{FG} asserts in the case of $L^{2m+3}(2^{\alpha(m)-1})$ that
\begin{equation}\label{previous2FG}
\TC(L^{2m+3}(2^{\alpha(m)-1}))\geq4m-2^{\nu(m)+2}-1
\end{equation}
provided $\alpha(m+1)\geq\alpha(m)$---i.e.~when~(\ref{topcases}) does not apply. But Corollary~\ref{cotainferior} improves the lower bound in~(\ref{previous2FG}) by $2^{\nu(m)+2}-3$ units.

\section{On the annihilator of the toral class in $ku_*(\zet_{2^k}\times\zet_{2^e})$}\label{S1}

Let $ku$ stand for the connective cover of complex $K$-Theory, and write $ku_*X$ for the reduced $ku$-homology of a space (or spectrum) $X$. The obvious projection $\zet \times \zet \rightarrow \zet_{2^k}\times\zet_{2^e}$ determines a stable map $S^2 \rightarrow B( \zet_{2^k}\times\zet_{2^e})$ and a corresponding ``toral'' class $\tau=\tau_{k,e}\in ku_2B(\zet_{2^k} \times \zet_{2^e})$. Note that $\tau$ actually lies in the $ku_*$-direct summand $ku_*B(\zet_{2^k})\otimes_{ku_*}ku_*B(\zet_{2^e})$. 

\medskip
The first result in this section (proved in Section~\ref{S4}) identifies key elements in Ann$_{ku_*}(\tau)$, the $ku_*$-annihilator ideal of $\tau$.

\begin{theorem}
\label{T1}
Let $v_1\in ku_2$ correspond to Bott periodicity. For $k\geq e$ and $\hspace{.4mm}1\leq j\leq e$ consider the elements $$\varepsilon_j =2^{e-j}v_1^{2^{j-1}(k-e+3)-2}\in ku_*.$$ Then $\varepsilon_j\cdot\tau_{k,e}=0$, for $j=0,\ldots,e$, where we set $\varepsilon_0 =2^e$.
\end{theorem}

Plenty of evidence (some of which is discussed in the following paragraphs) points toward the possibility that the elements $\varepsilon_j$ in Theorem~\ref{T1} generate Ann$_{ku_*}(\tau)$. For instance, Theorem~\ref{contribucion} below---a crucial ingredient in our proof of Theorem~\ref{CTC}---settles the initial steps in such a task. The following potential picture, which implies that the $\varepsilon_j$ would indeed generate Ann$_{ku_*}(\tau)$, arose from extensive computations mimicking those in the classical Conner-Floyd conjecture:

\begin{conjetura}
\label{T2}
Let $I=I_{k,e}$ be the ideal of $ku_*$ generated by the elements $\varepsilon_j$ for $j=0,\ldots, e$. Then there is a $ku_*$-filtration of $ku_*B(\zet_{2^k})\otimes_{ku_*} ku_*B(\zet_{2^e})$ whose associated graded object is $ku_*/I$-free, and has (the class of) $\tau_{k,e}$ as a basis element.
\end{conjetura}

\begin{nota}\label{motiv}{\em
An important motivation for Theorem~\ref{T1} and, specially, Conjecture~\ref{T2} comes from a desire of proving the corresponding statements with $ku$ replaced by the 2-primary Brown-Peterson spectrum (or, for that matter, by any $BP \langle n\rangle$ with $n\geq2$). This would yield a ($\zet_{2^k}\times\zet_{2^e}$)-analogue of the classical Conner-Floyd conjecture. An immediate consequence of such a potential result is that the $BP$-projective dimension of $B(\zet_{2^k}\times\zet_{2^e})$ would be $2\,$---the ($\zet_{2^k}\times\zet_{2^e}$)-case of an old conjecture of Landweber. (Some of these goals---for $k=e=2$, as well as for $e=1$---have been accomplished in Nakos' Ph.~D.~work~\cite{Nakos}.) Furthermore, on the applications side, and yet more interesting, is the fact that a proof of the $BP$-version of Conjecture~\ref{T2} would complete an important step toward proving the general case of Conjecture~\ref{conejura} (cf.~\cite{GZ}).
}\end{nota}

The starting point for the second result in this section comes from the observation (Proposition~\ref{PAx}) that, since $ku_*$ is polynomial on a {\em single} variable, the verification that the elements $\varepsilon_j$ in Theorem~\ref{T1} generate Ann$_{ku_*}(\tau)$ follows easily from the next conjecture:

\begin{conjetura}\label{C} None of the elements $2^{e-j}v_1^{2^{j-1}(k-e+3)-3}$ {\rm ($1\leq j \leq e$)} annihilates $\tau_{k,e}$.
\end{conjetura}

Conjecture~\ref{C} is true for $k\ge e=1$ and $k\ge e=2$. More generally, the following result is proved in Section~\ref{seco}:

\begin{theorem}\label{contribucion}
Neither of the elements $2^{e-1}v_1^{k-e}$ and $2^{e-2}v_1^{2(k-e)+3}$ annihilates the toral class $\tau_{k,e}$ \emph{(}the latter element makes sense only for $e\geq 2\hspace{.2mm}$\emph{)}. Consequently, the elements $\varepsilon_j$ in Theorem~\emph{\ref{T1}} generate Ann$_{ku_*}(\tau_{k,e})$ provided $e\leq 2$ and $k\geq e$.
\end{theorem}

\begin{nota}\label{aprox}{\em
Conjecture~\ref{C} for $j=3$ and $k=e$ and, therefore, the last sentence in Theorem~\ref{contribucion} for $k=e\leq3$ are proved in the Ph.~D.~thesis of the second author\footnote{The thesis extends in fact the range $k=e\leq3$ to $k=e\leq4$, for which Conjecture~\ref{T2} is also verified. This depends on extensive computer-based calculations.}. As a consequence, it is deduced that, for $k=e\geq4$, none of the elements
\begin{equation}\label{noneof}
2^{e-j}v_1^{11\cdot 2^{j-3}-2}\quad (4\leq j\leq e)
\end{equation}
annihilates $\tau_{e,e}$. Note that, for $j=4$, this says that $2^{e-4}v_1^{20}\tau_{e,e}\neq0$, involving a power of $v_1$ which is only one less than that described by Conjecture~\ref{C}. Although the {\it difference} between the exponents in $v_1$ in~(\ref{noneof}) and Conjecture~\ref{C} (for $k=e$) grows exponentially on $j$, it is its {\it quotient} the one with a good asymptotic behavior, as it tends to $\frac{11}{12}$ as $t\rightarrow\infty$. The proofs of these facts are given in Chapter~3 of~\cite{Velasco}, and are based on a much more elaborated argument than that given in Section~\ref{seco} to prove Theorem~\ref{contribucion}. Since we have no (say, TC-)application for these extended results, we leave the interested reader to look in~\cite{Velasco} for proof details of the facts remarked in this paragraph.
}\end{nota}

The rest of the section is devoted to preparing the grounds for the proof of Theorem~\ref{T1}---the actual proof is done in Section~\ref{S4}.

\medskip
We use the shorthand $ku_*(e)$ and $ku_*(k,e)$, respectively, for the reduced $ku_*$-homology of $\Sigma^{-1}B\zet_{2^e}$ and $\Sigma^{-2}B\zet_{2^k}\wedge B\zet_{2^e}$. The suspended spectra are taken just for notational convenience as bottom classes become zero-dimensional. It is well known that $ku_*(e)$ has a length-$1\;\; ku_*$-resolution
$$
0\longrightarrow L_1 \stackrel{\partial_e}{\longrightarrow}L_0 \longrightarrow ku_* (e) \longrightarrow 0
$$
where $L_1=L_0$ is the $ku_*$-free module on generators $z_i$ for $i\geq 0$. Each $z_i$ has dimension $2i$ and will also be denoted as $(i)$. The map $\partial _e$ comes from the $2^e$-series for the multiplicative formal group law $[2^e](x) = \frac{(1+v_1x)^{2^e}-1}{v_1}= \sum_{i=1}^{2^e}\binom{2^e}{i}v_1^{i-1}x^i$. Explicitly,
\begin{equation}
\label{Erds}
\partial_e (i) =\sum^{2^e-1}_{s\geq 0}a_s (i-s)
\end{equation}
where $(\ell )=0$ for $\ell <0$, and $a_s=\binom{2^e}{s+1}v^s_1$. 
In particular, the $ku_*$-K\"unneth spectral sequence for 
$\Sigma^{-2}B\zet_{2^k}\wedge B\zet_{2^e}$ collapses to the usual 
Landweber short exact sequence
$$
0\to ku_* (k) \otimes_{ku_*} ku_* (e) \to ku_*(k,e) \to \Sigma {\rm Tor}_1^{ku_*}(ku_* (k), ku_*(e))\to 0. 
$$

\smallskip
The attention in this and the next two sections focuses on the tensor group $ku_* (k) \otimes_{ku_*} ku_* (e)$---where the toral class lies. An important computational tool will be given by the Smith $ku_*$-morphism $\delta\colon ku_*(e)\rightarrow ku_*(e)$ determined by $\delta (i)=(i-1)$. We have the two endomorphisms $\delta_1=\delta \otimes 1$ and $\delta_2 =1\otimes \delta$ of $ku_* (k) \otimes ku_*(e)$ through which we can define a $T$-module structure on $ku_*(k) \otimes ku_* (e)$, where $T=ku_*[[\delta_1, \delta_2]]$ is the power series ring on two variables $\delta_1$ and $\delta_2$ with coefficients in $ku_*$. An element $(a)\otimes (b)\in ku_* (k) \otimes_{ku_*} ku_* (e)$ will simply be denoted by $(a,b)$. For instance, the toral class $\tau \in ku_0(k,e)$ corresponds to $(0,0)$. We will generically denote by $((s))$  any linear combination of elements of the form $(a,b)$ with $a+b=s$. 

\begin{proposition}
\label{P3} For $j=1, \ldots, e$ any element $2^{k-j} v_1^{2^j+2^{j-1}-2}(a,b)$ lies in the $T$-module generated by terms of the form
$$
2^{k-s}v^{2^s+2^{j-1}-2}_1((a+b+2^j -2^s)),\quad {\rm with}\quad j<s\leq k.
$$
\end{proposition}

The proof of Proposition~\ref{P3} and the deduction of Theorem~\ref{T1} from Proposition~\ref{P3} are given in the next section. It will be convenient to write
\begin{equation}
\label{Fml}
2^{k-j}v^{2^j+2^{j-1}-2}_1 (a,b) \equiv \sum^k_{s=j+1}2^{k-s}v^{2^s +2^{j-1}-2}_1 ((a+b+2^j -2^s))
\end{equation}
for expressing the conclusion in Proposition~\ref{P3}. Here, the congruence symbol is to be read as ``modulo $T$-multiples of the elements on the right hand side''. This notation will be in force through the following sections.

\begin{nota}\label{concluding}{\em
Although the proof of Proposition~\ref{P3} is relatively straightforward, and has the $k=e$ case of Theorem~\ref{T1} as an obvious consequence, in Section~\ref{S4} we need to appeal to a rather involved process in order to derive the general case of Theorem~\ref{T1} out of Proposition~\ref{P3}. In a sense, our approach to Theorem~\ref{T1} fills in the gap for the case $k>e$. The other major feature of our proof for Theorem~\ref{T1} has already been discussed in Remark~\ref{motiv}, namely, the possibility of extending this result to the $BP$-case. In this respect, it is to be observed that Proposition~\ref{P3} claims, in particular, that not only the toral class, but the whole tensor product $ku_*(k)\otimes ku_*(e)$ is killed by $v_1^{2^{e-1}(k-e+3)-2}$ when $k=e$. This situation might as well hold for any $k\geq e$, but perhaps not with $BP$. Thus, in an eventual $BP$-generalization of these results, the present direct form of Proposition~\ref{P3} might need to be replaced by the type of (inductive) methods in Section~\ref{S4}.
}\end{nota}

\section{Proof of Theorem~\ref{T1}}
\label{S4}
The arithmetical manipulations in this section are based on the standard fact that the highest power of $2$ dividing the binomial coefficient $\binom{2^e}{s+1}$ is
\begin{itemize}
\item[{(i)}] \ equal to $e-\ell$, if $s+1=2^\ell$, for some $\ell=0, \ldots, e$;
\item[{(ii)}] \ grater than $e-\ell$, if $2^\ell<s+1<2^{\ell+1}$, for some $\ell=0, \ldots, e-1$.
\end{itemize}
For instance, in terms of the notation set up in~(\ref{Fml}), the relation imposed by~(\ref{Erds}) on the second tensor factor of $ku_*(k) \otimes ku_*(e)$ yields $$2^e(a,b)\equiv \sum\limits^{e}_{s=1}2^{e-s}v_1^{2^s-1}(a,b-2^s+1),$$ a ($j=0$)-version of~(\ref{Fml}) when $k=e$ as long as we think of $2^{-1}$ as being~$1$. The proof of Proposition~\ref{P3} will make a systematic use of analogous considerations based on a suitable combination of the relations coming from both tensor factors.

\begin{proof}[Proof of Proposition~\ref{P3}]
Set $g(\ell)=2^\ell-1$ and, for a generator $(a,b)$ of $ku_*(k)\otimes ku_*(e)$, write $2^{k-j+1}v^{g(j-1)}_1 (a) = 2^{k-j+2}A+A'$ and $2^{e-j}v^{g(j)}_1(b) =  2^{e-j+1}B+ B'$ with 
$$
A\equiv\sum^{j-2}_{s=0}2^{j-s-2}v^{g(s)}_1 (a+g(j-1)-g(s)),\;\;B\equiv\sum^{j-1}_{s=0} 2^{j-s-1}v^{g(s)}_1 (b+g(j)-g(s)),
$$

\vspace{-4mm}
$$
A'\equiv\sum^k_{s=j}2^{k-s}v^{g(s)}_1(a+g(j-1)-g(s)),
\;\;B'\equiv\sum^e_{s=j+1}2^{e-s}v^{g(s)}_1 (b+g(j)-g(s)).
$$
Note that the terms coming from (ii) above are meant to be taken care of by the suitable use of the congruence symbols (of course, $B'=0$ when $j=e$, and $A=0$ when $j=1$). Then, with $k=e+d$, $d\geq 0$ we get
\begin{eqnarray*}
\lefteqn{2^{k-j}v_1^{g(j)+g(j-1)}(a,b) \;\;=\;\; 2^dv_1^{g(j-1)}\left(a, 2^{e-j}v_1^{g(j)}(b)\right)}\\
&\ \ \equiv & 2^d v_1^{g(j-1)}\Big( a, 2^{e-j+1}B+\sum^e_{s=j+1}2^{e-s}v^{g(s)}_1 (b+g(j)-g(s))\Big)\\
&\ \ =& 2^{k-j+1}v^{g(j-1)}_1 (a,B) +\sum^e_{s=j+1} 2^{k-s}v^{g(s)+g(j-1)}_1 (a, b+g (j)-g(s)).
\end{eqnarray*}
The summation in the previous line has the form required in~(\ref{Fml}). Thus we only need to deal with the term $2^{k-j+1}v^{g(j-1)}_1 (a,B)$ which is congruent with
\begin{equation}
\label{E2}
2^{k-j+2}(A,B)+\sum^k_{s=j} 2^{k-s}v^{g(s)}_1(a+g(j-1)-g(s), B).
\end{equation}
The last summation is congruent to
$$
\sum^k_{s=j}2^{k-s}v^{g(s)}_1 \Big(a+g(j-1)-g(s), \sum^{j-1}_{\sigma =0} 2^{j-\sigma -1} v^{g(\sigma)}_1 (b+g(j)-g(\sigma))\Big)
$$
or, in the proposed shortened form,
\begin{equation}\label{repetida}
\;\sum^k_{s=j} \sum^{j-1}_{\sigma =0} 2^{k+j-s-\sigma -1}v^{g(s)+g(\sigma)}_1((a+b+g(j)+g(j-1)-g(s)-g(\sigma)))
\end{equation}
for each of whose summands one has $1\leq s+\sigma +1-j\leq k$. Those  with $s+\sigma +1-j>j$ are ($T$-multiples of terms) of the form required in~(\ref{Fml}) since, in that case, $g(s)+g(\sigma)\geq g(s+\sigma +1-j)+g(j-1)$, as it can easily be verified. Those with $s+\sigma +1-j=j$ take in fact the form 
$2^{k-j}v^{g(s)+g(\sigma)}_1 (\alpha, \beta)$, where now 
$g(s)+g(\sigma)\geq g(j)+g(j-1)$, $\alpha <a$, and $\beta >b$. Finally, those with $s+\sigma +1-j=\mu$ for a fixed $1\leq \mu <j$ and necessarily with $\mu\geq \sigma +1$ are taken into account within
$$
\sum^{\mu -1}_{\sigma =0} 2^{k-\mu}v^{g(\mu +j-\sigma -1)+g(\sigma)}_1 ((\cdots\!\hspace{.8mm}))
$$
where numbers inside double parenthesis are forced by dimensional reasons.
But since $\mu<j$, all terms in the last summation are $T$-multiples of that with $\sigma =\mu-1$. We deduce that the summation in~(\ref{E2}) is congruent to {\small$$
\sum^k_{s>j}\!\hspace{.6mm}2^{k-s}v_1^{g(s)+g(j-1)}((\cdots\!\hspace{.8mm}))
+\sum^{j-1}_{\mu =1}\!\hspace{.4mm}2^{k-\mu}v^{g(j)+g(\mu-1)}_1 
((\cdots\hspace{.7mm}))+\!\!\!\!\sum_{(\alpha,\beta)\in\Lambda}\!\!\!\!
2^{k-j}v^{g(j)+g(j-1)}_1 (\alpha, \beta)
$$}\noindent where $\Lambda\subseteq\{(\alpha,\beta)\;|\;\alpha+\beta=a+b,\;\,\alpha<a,\;\,\beta>b\}$. On the other hand, the first term in~(\ref{E2}) is 
\begin{eqnarray*}
\lefteqn{2^{k-j+2}(A,B)=2^{d+1}(A, 2^{e-j+1}B)\equiv 2^{d+1}
\biggl(\!A, \sum^e_{s=j}2^{e-s}v_1^{g(s)} (b+g(j)-g(s)) \biggl)}\\
&=& \!\!\sum^e_{s=j}2^{k-s+1} v^{g(s)}_1(A, b+g(j)-g(s))\\
&\equiv & \!\!\sum^e_{s=j} 2^{k-s+1}v^{g(s)}_1 \biggl(\sum^{j-2}_{\sigma =0} 2^{j-\sigma -2} v^{g(\sigma)}_1 (a+g(j-1)-g(\sigma)), b+g(j) -g(s)\biggr)\\
&=& \!\!\sum^e_{s=j} \sum^{j-2}_{\sigma =0} 2^{k-1+j-s-\sigma}v^{g(s)+g(\sigma)}_1((a+b+g(j)+g(j-1)-g(s)-g(\sigma)))
\end{eqnarray*}
which is similar to~(\ref{repetida}) and, therefore, the corresponding analysis applies, with the exception that, in the case $s+\sigma+1-j=j$, the relations $\alpha<a$ and $\beta>b$ might not necessarily hold but, instead, one gets the strict inequality $g(s)+g(\sigma)>g(j)+g(j-1)$ from the new restriction $\sigma\leq j-2$.

\smallskip
Putting everything together, there results that $2^{k-j} v_1^{g(j)+g(j-1)}(a,b)$ 
is congruent to
\begin{equation}\label{E3}\begin{array}{l}
\displaystyle\sum\limits^k_{s>j} 2^{k-s}v^{g(s)+g(j-1)}_1 ((\cdots\!\hspace{.8mm}))
\;+\;\sum\limits^{j-1}_{\mu =1}2^{k-\mu}v^{g(j)+g(\mu-1)}_1 ((\cdots\!\hspace{.8mm}))\\
{}\;+\;\sum\limits_{R>0} 2^{k-j} v^{g(j)+g(j-1)+R}_1 ((\cdots\!\hspace{.8mm}))
\;+\!\sum\limits_{(\alpha,\beta)\in\,\Lambda}\!\!2^{k-j}v^{g(j)+g(j-1)}_1 (\alpha, \beta).
\end{array}\end{equation}
Now, for $1\leq \mu \leq j-1$ (holding only for $j>1$) one has $g(j)+g(\mu-1)>g(\mu) +g(\mu -1)$ so that, by induction, each term on the second summation in~(\ref{E3}) becomes 
$$
2^{k-\mu}v_1^{g(j)+g(\mu-1)}((\cdots\!\hspace{.8mm}))\equiv \sum_{\lambda >\mu}2^{k-\lambda}v_1^{g(\lambda)+g(j)+g(\mu-1)-g(\mu)}((\cdots\!\hspace{.8mm})).
$$
Here, terms with $\lambda >j$ are easily seen to be contained (up to congruence) in the first summation 
of~(\ref{E3}), whereas those with $\lambda =j$ are contained in the third summation 
of~(\ref{E3}). But the terms $2^{k-\lambda}v_1^{g(\lambda )+g(j)+g(\mu-1)-g(\mu)}((\cdots))$ with $1\leq \mu <\lambda \leq j-1$ are easily seen to be $T$-multiples of the corresponding $\lambda$-th term in the second summation of~(\ref{E3}). Therefore, an auxiliary inductive process on $\mu=1,\cdots,j-1$ allows us to get rid of the second summation in the expression~(\ref{E3}) for $2^{k-j}v_1^{g(j)+g(j-1)}(a,b)$. Then, by iterating the resulting formula, we can also get rid, first, of the last summations in~(\ref{E3}) and, then, of the third summation in~(\ref{E3}).
\end{proof}

One further formula is needed before proving Theorem~\ref{T1}. We use the relation imposed by~(\ref{Erds}) on the second tensor factor to write
\begin{equation}\label{dereferencia}
2^{e-j}v^{g(j)}_1(a,b)\equiv \sum_{s\in I_j}2^{e-s}v^{g(s)}_1 ((a+b+g(j)-g(s)))
\end{equation}
for $j=1,\ldots,e$, where $I_j=\{s\;|\;0\leq s\leq e,\;\, s\neq j\}$. Then, for $j<\ell\leq d+j$ (recall $d=k-e\geq 0$) we have
\begin{eqnarray}
\label{E5}
\lefteqn{2^{k-\ell}v^{2^{j-1}(\ell +3-j)-2}_1(a,b) \;\;=\;\; 2^{d+j-\ell}v^{2^{j-1}(\ell +1-j)-1}_1\cdot 2^{e-j}v^{g(j)}_1(a,b)\nonumber}\\
&\ \hspace{1.2cm} \equiv & 2^{d+j-\ell}v^{2^{j-1} (\ell +1-j)-1}_1
\biggl(\,\sum_{\,s\in I_j}2^{e-s}v^{g(s)}_1((a+b+g(j)-g(s)))\biggr)\nonumber\\
&\ \hspace{1.2cm}=& \sum_{s\in I_j}2^{k+j-\ell-s}v_1^{2^{j-1}(\ell -j)+g(j-1)+g(s)}((a+b+g(j)-g(s))).
\end{eqnarray}

Note that, in the last summation, $1\leq \ell +s-j\leq k$. In particular, we could use Proposition~\ref{P3} in order to get rid of some such summands (those with a high enough power of $v_1$). However the resulting expression seems to become unnecessarily complicated. Instead, all summands in~(\ref{E5}) will be taken care of (in the arguments below) by means of a suitable inductive process.

\medskip
We now start working toward the proof of Theorem~\ref{T1}.
The relation $2^e(0,0)=0$ is obvious as it comes directly from the second tensor factor. In fact, in $ku_*(e)$ one can easily prove (see for instance Corollaries 2.6 and 2.7 in~\cite{cfimm})
\begin{equation}
\label{ord}
2^{e+i}(i)=0\quad {\rm and}\quad 2^{e+i-1}v_1(i)=2^{e+i}(i+1),\quad {\rm for}\quad i\geq 0.
\end{equation}
This implies the ($j=1$)-case in the conclusion of Theorem~\ref{T1}:
\begin{equation}
\label{EQ}
\begin{array}{c}\displaystyle
2^{e-1}v^{k-e+1}_1(0,0)=2^ev^{k-e}_1(0,1)\hspace{3cm}\;\\
\rule{0mm}{5mm}=2^{e+1}v^{k-e-1}_1 (0,2)=\cdots =2^k (0, d+1)=0
\end{array}
\end{equation}
where the last equality comes directly from the bottom relation in the first tensor factor. However, the ($j>1$)-cases in the conclusion of Theorem~\ref{T1} are much more cumbersome to derive. The auxiliary constructions below are intended to organize an elaborated process whose main idea is to use~(\ref{E5}) as a generalization of~(\ref{EQ}) in order to fill in the gap left by~(\ref{Fml}) when $k>e$ (this is a slightly more detailed description of the first two sentences in Remark~\ref{concluding}). Thus, throughout the rest of this section we assume $k=e+d$ with $d>0$.

\begin{definition}
\label{D4} {\em Set $p(0)=-1$ and $p(\sigma)=(d+1)g(\sigma -1)$ for $\sigma\geq1$. Consider the set $J$ consisting of all pairs of non-negative integers $(i,j)$ satisfying $1\leq j\leq d+\sigma$ whenever $p(\sigma -1)<i\leq p(\sigma)$ with $1\leq \sigma \leq e$. The number $\sigma$ above is determined by $i$; yet at times we will denote it as $\sigma(i,j)$. Then, for $(i,j)\in J$, set
$$
s(i,j)=\left\{
\begin{array}{lcl}
d+1+i-2^{j-1}, &{\rm for} & j\leq \sigma(i,j),\\
d+1+i-2^{\sigma(i,j)-1}(j-\sigma(i,j)+1), & {\rm for} & j\geq \sigma(i,j),
\end{array}\right.
$$
$$
u(i,j)=\left\{
\begin{array}{lcl}
g(j)+g(\sigma(i,j)-1), & {\rm for} & j\leq \sigma(i,j),\\
2^{\sigma(i,j)-1}(j+3-\sigma(i,j))-2, &{\rm for} & j\geq \sigma(i,j),
\end{array}\right.
$$
and
$$
E(i,j)=2^{k-j}v^{u(i,j)}_1((s(i,j))).
$$
}\end{definition}

Note that, for $1\leq\sigma\leq e$, $s(p(\sigma), d+\sigma)=0$ and $u(p(\sigma), d+\sigma)=2^{\sigma -1}(d+3)-2$, so that $E(p(\sigma), d+\sigma)$ reduces to the term  $\varepsilon_\sigma\cdot (0,0)$ in Theorem~\ref{T1}. Thus, Theorem~\ref{T1} is a particular case of the following more complete result:

\begin{theorem}
\label{T1bis} $E(i,j)=0$, for $(i,j)\in J$.
\end{theorem}

The main computational task in this section is to show that, for any $(i,j)\in J$ and in terms of the convention set up in~(\ref{Fml}),
\begin{equation}
\label{G}
E(i,j)\equiv \sum E(i', j') 
\end{equation}
where the summation runs over (some of) those $(i', j')\in J$ which, in the lexicographic order of $J$, satisfy $(i', j')<(i,j)$. Theorem~\ref{T1bis} will then be a consequence of:

\begin{lema}
\label{L5}
$E(0,1)=0$.
\end{lema}
\begin{proof}
$E(0,1)=e^{k-1}v_1 ((d))\equiv\sum^k_{s\geq 2}2^{k-s}v^{g(s)}_1 ((d+2-2^s))$ in view of~(\ref{Fml}). But, in view of the first relation in~(\ref{ord}), each of these summands is trivial as $d-2^s+2+e\leq k-s$ for $s\geq 2$.
\end{proof}

The remainder of this section is devoted to establishing~(\ref{G}). As a way of example, we first complete the analysis for the case $i=0$ (that is $\sigma =1$) started in the proof of Lemma~\ref{L5}. Thus, for $1<j\leq d+1$, (\ref{E5}) yields
$$
E(0,j) =2^{k-j}v^j_1 ((d+1-j))=\sum_{s\in I_1} 2^{k+1-j-s}v^{j-1+g(s)}_1((d+3-j-2^s)),
$$
where $I_1$ is defined in~(\ref{dereferencia}). The term with $s=0$ corresponds to $E(0,j-1)$ whereas, as in the proof of Lemma~\ref{L5}, the terms with $s\geq 2$ are trivial in view of~(\ref{ord}).

\begin{proof}[Proof of~\emph{(\ref{G})} in the general case]
Suppose $p(\sigma -1)<i\leq p(\sigma)$, with $\sigma \geq 2$, and consider $E(i,j)$ as given in Definition~\ref{D4} (so that $1\leq j\leq d+\sigma)$.

\vglue .25cm
\noindent
{\bf Case $1\leq j\leq \sigma$}. From~(\ref{Fml}) we have
\begin{eqnarray*}
E(i,j)&=&2^{k-j}v^{g(j)+g(\sigma -1)}_1 ((d+1+i-2^{j-1}))\\
&=&\sum^k_{s>j} 2^{k-s}v^{g(s)+g(\sigma -1)}_1 ((d+1+i+2^{j-1}-2^s)).
\end{eqnarray*}
The summation really runs for $j<s<d+\sigma$ since, when $s\geq d+\sigma$ (and given that $i\leq p(\sigma)$ and $j\leq \sigma$), one actually has $d+1+i+2^{j-1}-2^s<0$. But in this restricted range, one easily checks that the $s$-th summand in the last summation is a $T$-multiple of $E(i-1,s)$. Of course such a fact has to be verified by dividing into cases, depending on whether $s\geq \sigma$ or $s\leq \sigma$, and whether $i-1>p(\sigma -1)$ or $i-1=p(\sigma -1)$. The actual verifications are left to the reader and, for the sake of illustration, we only sketch a representative situation. When $\sigma \leq s$ and $i-1>p(\sigma -1)$ one has
$$
E(i-1, s) =2^{k-s}v^{2^{\sigma -1}(s+3-\sigma)-2}_1 ((d+i-2^{\sigma -1}(s-\sigma +1)))
$$
so that we only need to verify that $d+1+i+2^{j-1}-2^s\leq d+i-2^{\sigma -1}(s-\sigma +1)$ (for which it is convenient to keep in mind that, in the present situation, $j\leq \sigma \leq s$ with $j<s$) and that $2^{\sigma -1}(s+3-\sigma)-2\leq g(s)+g(\sigma -1)$.
\vglue .25cm
\noindent
{\bf Case $\sigma < j\leq d+\sigma$}. Using~(\ref{E5}) we now have
\begin{eqnarray*}
\lefteqn{E(i,j)= 2^{k-j} v^{2^{\sigma-1}(j+3-\sigma)-2}_1((d+1+i-2^{\sigma-1}
(j-\sigma +1)))}\\
&\!\!\!\!=&\hspace{-3mm}\sum_{s\in I_\sigma}2^{k+\sigma-j-s}
v^{2^{\sigma-1}(j-\sigma)+g(\sigma-1)+g(s)}_1((d\hspace{.4mm}{+}\hspace{.4mm}1\hspace{.4mm}{+}\hspace{.4mm}i\hspace{.4mm}{-}\hspace{.4mm}2^{\sigma-1}(j\hspace{.4mm}{-}\hspace{.4mm}\sigma\hspace{.4mm}{+}\hspace{.4mm}1)\hspace{.4mm}{+}\hspace{.4mm}2^\sigma\hspace{.4mm}{-}\hspace{.4mm}2^s)).
\end{eqnarray*}
This time the summation is relevant only for $j+s-\sigma \leq d+\sigma -1$ because, otherwise, the summands have a (({\it negative}))-part. One then verifies (again taking into consideration suitable cases, so that Definition~\ref{D4} is applied correctly) that summands with $s<\sigma$ (so that $1\leq j+s-\sigma <j$) are $T$-multiples of $E(i, j+s-\sigma)$, whereas summands with $s>\sigma$ are $T$-multiples of $E(i-1, j+s-\sigma)$.
\end{proof}

\section{Proof of Theorem~\ref{contribucion}}\label{seco}
In this section $k\geq e$ are fixed, and attention is focused on the $ku_*$-annihilator ideal of the toral class $\tau=\tau_{k,e}\in ku_2B(\mathbb{Z}_{2^k}\times\mathbb{Z}_{2^e})$.
\begin{proposition}\label{PAx}
Let $2^{p_i}v_{1}^{q_i}$, $i=1,\dots,T$, be elements annihilating $\tau$ such that no $2^{p_i}v_{1}^{q_i-1}$ annihilates $\tau$, $i=2,3,\dots,T$. If $p_{i+1}=p_i-1$, $\,p_T=0$, and $\,0=q_1<q_2<\cdots<q_T$, then in fact $\mathrm{Ann}(\tau)=(2^{p_i}v_{1}^{q_i})_{i=1,\dots,T}$.
\end{proposition}

\begin{proof}Let $f\in \mathrm{Ann}(\tau)$ be a 
homogeneous element. Since $2^{p_1}\tau=0$, we can assume $f=2^mv_{1}^n$.
If $m\geq p_1$, then evidently $2^mv_{1}^n \in (2^{p_i}v_{1}^{q_i})_{i=1,\dots,T}$. 
Assume $m=p_\ell$, $\ell=2,\dots,T$. If $n<q_\ell$, 
then $2^{p_\ell}v_{1}^{q_\ell-1}$ is a multiple of
$2^mv_{1}^n \in \mathrm{Ann}(\tau)$ and therefore we would have
$2^{p_\ell}v_{1}^{q_\ell-1}\in \mathrm{Ann}(\tau)$, which contradicts the hypothesis.
Thus we must have $n\geq q_\ell$, so that $f=2^mv_{1}^n$ is a multiple of
$2^{p_\ell}v_{1}^{q_\ell}$, and evidently $f\in (2^{p_i}v_{1}^{q_i})_{i=1,\dots,T}$.
\end{proof}

This section's strategy is to work directly in the free $ku_*$-module 
$F$ with basis the pairs $(i,j)$, $i,j \geq 0$, and observe that 
Conjecture~\ref{C} can be proved 
by checking that it is not possible 
to have in $F$ relations of the form 
\begin{equation}\label{F4}
2^{e-j}v_1^{f(j)-1}(0, 0) =
\sum_{\alpha,\beta}c_{\alpha,\beta}g_{\alpha,\beta} +
\sum_{\alpha,\beta}d_{\alpha,\beta}h_{\alpha,\beta},\quad 
c_{\alpha,\beta}, d_{\alpha,\beta}\in\mathbb{Z}
\end{equation}
where $f(j)=2^{j-1}(k-e+3)-2$,
\begin{equation}\label{F5}
g_{\alpha,\beta} = \binom{2^e}{1}v_1^{f(j)-\alpha-\beta-1}(\alpha,
\beta) + \binom{2^e}{2}v_1^{f(j)-\alpha-\beta}(\alpha, \beta-1) +
\cdots,
\end{equation}
\begin{equation}\label{F6}
h_{\alpha,\beta} = \binom{2^k}{1}v_1^{f(j)-\alpha-\beta-1}(\alpha,
\beta) + \binom{2^k}{2}v_1^{f(j)-\alpha-\beta}(\alpha-1, \beta) +
\cdots,
\end{equation} and where the summations in~(\ref{F4}) run over indexes
$\alpha$ and $\beta$ subject to
\begin{equation}\label{F7}
\alpha, \beta \geq 0,\,\, \alpha+\beta < f(j).
\end{equation}
The 2-divisibility properties of binomial coefficients will play a crucial role in our arguments, and we record for future reference a well-known result stating the form in which we will make use of these 2-divisibility properties.

\begin{lema}\label{L1}
The functions $\alpha$ and $\nu$ defined respectively in~\emph{(\ref{topcases})} and~\emph{(\ref{previous2FG})} satisfy $\nu\binom{a}{b}=\alpha(b)+ \alpha(a-b)- \alpha(a)$. In particular, if $a,b,c$ are non-negative integers with $b\leq a$ and $1\leq c < 2^{a-b+1}$, the binomial coefficient $\binom{2^a}{c}$ is divisible by $2^b$ \emph{(}a preciser form of this assertion---not needed in the sequel---has already been noted at the beginning of Section~\emph{\ref{T1}}\emph{)}.
\end{lema}

\medskip
Throughout the 
rest of the section we will assume, to reach a contradiction, that
an equation of the form~(\ref {F4}) has been given with $j\leq2$.
The next two propositions will be central.

\begin{proposition}\label{P1}
\begin{equation}
\nu(c_{\alpha,\beta})\geq \left\{\begin{array}{ll} \alpha+\beta, &
\textrm{if}\quad j=1,
\\ \alpha+\beta-d-3, & \textrm{if} \quad j = 2. \nonumber
\end{array}\right.
\end{equation}
\end{proposition}

\begin{proof} {\bf Case $j=1$}. We proceed by inverse induction on
$\alpha +\beta$. Notice that, in agreement with (\ref {F7}), the
range of interest is $0\leq\alpha +\beta\leq d$ (we keep the notation
$d=k-e$ introduced in the previous section). Let us focus
attention on a basis element $(\alpha, \beta)$ in $F$ 
with $\alpha + \beta=d$. In
these conditions the terms $g _{\alpha,\beta}$ and $h
_{\alpha,\beta}$ are the only ones contributing in~(\ref{F4})
with multiples of the basis element $(\alpha,\beta)$. The
contributions are $2^ec _{\alpha,\beta}$ and $2^kd
_{\alpha,\beta}$, respectively. Therefore, we should have $\nu (c
_{\alpha,\beta}) \geq k-e$.

\smallskip
We now assume $\nu (c _{\alpha',\beta'}) \geq \alpha'
+\beta'$ whenever $\alpha'+\beta'>p$ for some $p <d$, and demonstrate that 
$\nu(c_{\alpha,\beta})\geq p$ given $\alpha$ and $\beta$ with $\alpha
+\beta=p$. Again, we focus attention on the basis element $(\alpha,\beta)$
in $F$, and the way it can arise on the right hand side of~(\ref{F4}). 
By virtue of~(\ref {F6}), the contributions from the second summation
in~(\ref {F4}) must arise from summands having an index of the form
$(r,\beta)$, $r\geq \alpha$. For these summands we have
$$
\nu\binom{2^k}{r-\alpha+1}\geq p+e
$$
due to Lemma~\ref{L1} as
$r-\alpha+1=r +\beta-p+1\leq d-p+1 <2^{d-p+1}$ and $p+e\leq k\,$---the latter 
inequality holds because $p =\alpha +\beta\leq d=k-e$.
On the other hand, the contributions that arise from the
first summation in~(\ref{F4}) come from indexes of the form
$(\alpha, s)$, with $s\geq\beta$, and by induction satisfy the inequality
$\nu(c_{\alpha,s}) + \nu\binom{2^e}{s-\beta+1}\geq p+e$ for $s>
\beta$. Indeed,
$$
\nu\binom{2^e}{s-\beta+1}\geq p+e-\alpha-s \nonumber
$$
holds because the conditions of Lemma~\ref{L1}
are satisfied: $p+e-\alpha-s\leq e\,$ since $p =\alpha +\beta
<\alpha+s$, while $s-\beta+1 <2^{\alpha+s-p+1}$
since $s-\beta+1=s +\alpha-p+1 <2^{\alpha+s-p+1}$.
Therefore, when $s =\beta$ we must also have
$\nu(c_{\alpha, \beta}) + \nu\binom{2^e}{1} \geq p+e$; that is
\begin{eqnarray}\label{G3}
\nu(c_{\alpha,\beta})\geq p=\alpha+\beta.\nonumber
\end{eqnarray}

\medskip\noindent {\bf Case $j=2$}. 
We proceed by inverse induction  on
$\alpha +\beta$. This time the range under consideration
is $\,0\leq\alpha+\beta<f(2)
=2d+4$. So, to ground the induction, assume $\alpha +\beta=2d+3$. 
As in previous cases, we focus 
attention on the basis element $(\alpha, \beta)$, and the way it can arise 
on the right hand side of~(\ref{F4}). Only
$g_{\alpha,\beta}$ and $h_{\alpha,\beta}$ are relevant, and the
corresponding contributions to the summand $(\alpha,\beta)$ are
$2^ec_{\alpha,\beta}$ and $2^kd_{\alpha,\beta}$.
Therefore, we must have $\nu(c_{\alpha,\beta}) \geq k-e$.

\smallskip Now, we assume that $\nu(c_{\alpha',\beta'}) \geq \alpha'
+\beta'-d-3$ whenever $\alpha' +\beta'> p$ for some $p <2d+3$, and demonstrate
that $\nu(c_{\alpha,\beta}) \geq p-d-3$ when $\alpha +\beta=p$. We
focus attention on the basis element $(\alpha,\beta)$ in $F$. In view
of~(\ref{F6}), the contributions of the second summation in~(\ref
{F4}) must arise from summands with index of the form $(r,\beta)$,
$r\geq \alpha$. For these we have
$$
\nu\binom{2^k}{r-\alpha+1}\geq p+e-d-3
$$
because of Lemma~\ref{L1}
in virtue that $p+e-d-3\leq k$ since $p =\alpha +\beta\leq 2d+3=2
(k-e)+3$ and $r-\alpha+1=r +\beta-p+1\leq 2d+3-p+1<2^{2d-p+4}$.
On the other hand, the contributions from the first
summation of~(\ref {F4}) arise from indexes of the form $(\alpha,
s)$, with $s\geq \beta$, and by induction satisfy the inequality $\nu(c_{\alpha,
s}) + \nu\binom{2^e}{s-\beta+1} \geq p+e-d-3$ for $s> \beta$. Indeed,
$$
\nu\binom{2^e}{s-\beta+1}\geq p+e-\alpha-s
$$
holds because the conditions of Lemma~\ref{L1} are
satisfied: $p+e-\alpha-s\leq e$ since $p =\alpha +\beta\leq
\alpha+s$, while $s-\beta+1 <2^{\alpha+s-p+1}$
since $s-\beta+1 =\alpha+s-p+1 <2^{\alpha+s-p+1}$.
Therefore, when $s =\beta$ we must also have
$\nu(c_{\alpha,\beta}) + \nu\binom{2^e}{1} \geq p+e-d-3$, that is
$\nu(c_{\alpha,\beta})\geq p-d-3=\alpha+\beta-d-3$.
\end{proof}

\begin{proposition}\label{P2}
In the case $j=2$, the inequality of Proposition~\ref{P1} can be
improved to $\nu(c_{\alpha,\beta})\geq \alpha+\beta+i-d-3$ when the
following conditions are fulfilled:
\end{proposition}
\begin{itemize}
\item[{(a)}] \ $d+4-i \leq \alpha+\beta \leq 2d-2i+4$.
\item[{(b)}] \ $\alpha \leq d+3-i$.
\item[{(c)}] \ $1 \leq  i \leq d$.
\end{itemize}

\begin{proof} We proceed by double induction, first on $i$ in a range
limited by (c), and then on $\alpha +\beta$ in the range marked by
(a). Let us consider a triple $(i,\alpha,\beta)$ satisfying
(a), (b), and (c), and assume inductively that the
proposition has been verified for any other triple
$(i',\alpha',\beta')$ that, beside fulfilling (a), (b), and (c),
satisfies one of the following conditions:
\begin{itemize}
\item[{(d)}] \ $i' < i $,
\item[{(e)}] \ $i'=i\,$ \textrm{ and } $\,\alpha'+\beta'<\alpha+\beta$.
\end{itemize}
Of course, the induction hypotheses associated with (d) and (e)
are empty at the start of the corresponding induction. 
We note that the three
conditions (a), (b), and (c) imply $\beta \geq 1$ and $\alpha+\beta \geq 4$.
In particular, $(\alpha,\beta-1)$ and $(0,0)$ are
different basis elements and, consequently, the total coefficient
$T$ with which $(\alpha,\beta-1)$ appears on the right hand side
of~(\ref{F4}) must be null. Then, in view
of~(\ref{F5}), the conclusion we need in order to close the
induction translates into verifying that the contribution to $T$
from the summand with index $(\alpha,\beta)$ in the first summation
of~(\ref{F4}) is divisible by
\begin{equation}\label{F24}
2^{\alpha+\beta+i+e-d-4}
\end{equation}
(notice that in view of (a), the exponent in (\ref{F24}) is at
least $e$). But since $T=0$, it suffices to verify that the rest of the
contributions to $T$ from the right hand side of~(\ref{F4}) are also divisible by~(\ref{F24}).

\medskip
By virtue of (\ref{F6}), the contributions of the second
summation in~(\ref{F4}) must arise from summands with index of the
form $(r,\beta-1)$, $r\geq \alpha$. For these it will be enough to
verify that $\binom{2^k}{r-\alpha+1}$ is divisible by~(\ref{F24}).
But this will be a consequence of Lemma~\ref{L1}, as soon as we
argue the inequalities
\begin{itemize}
\item[(f)] $\alpha+\beta+i+e-d-4 \leq k$,
\item[(g)] $r-\alpha+1 < 2^{k+d-\alpha-\beta-i-e+5}$.
\end{itemize}
The first condition is equivalent to $\alpha +\beta \leq
2d-i +4$ and, therefore, it is guaranteed by~(a). On the other
hand, in case of the index $(r, \beta-1)$, the second inequality 
in~(\ref{F7}) asserts that $r-\alpha+1 \leq 2d-\alpha-\beta+5$, and
since $2^m \geq 2m$ for every $m\in\mathbb {Z}$, (g) will follow
as soon as the inequality 
\begin{equation}\label{sacala}
2d-\alpha-\beta+5 <2 (k-\alpha-\beta-i-e+d+5)
\end{equation}
is justified. But an elementary arithmetic manipulation shows that the latter
inequality is equivalent to the second inequality in (a).

\smallskip It remains to consider the contributions coming from summands
in the first summation of~(\ref{F4}) having index of the
form $(\alpha, s)$, with $s \geq \beta-1$, $s\neq \beta$. For these
we have to verify
\begin{equation}\label{F25}
\nu(c_{\alpha,s})+\nu\binom{2^e}{s-\beta+2}\geq
\alpha+\beta+i+e-d-4.
\end{equation}
The case $s =\beta-1$ follows from the case (e) of the induction,
except for when $\alpha +\beta=d+4-i\,$---corresponding to the
beginning of the induction---in which case~(\ref{F25}) is obvious.
Then, throughout the rest of the proof we will assume
\begin{eqnarray}\label{F26}
s > \beta.
\end{eqnarray} Likewise, when $i=1$, Proposition~\ref{P1} implies
that~(\ref{F25}) is a consequence of the inequality $\nu\binom{2^e}{
s-\beta+2} \geq \beta+e-s$, which in turn follows easily from
Lemma~\ref{L1}. Thus we will also assume
\begin{equation}\label{F27}
i \geq 2.
\end{equation}
Furthermore, when $\alpha+s=2d+3$ (recall that the top limit
for $\alpha+s$ is set by (\ref{F7})), Proposition~\ref{P1}
asserts that (\ref{F25}) is a consequence of the inequality
$\nu\binom{2^e}{s-\beta+2} \geq \alpha +\beta+i+e-2d-4$, which in turn
follows from Lemma~\ref {L1}. Indeed, the conditions
\begin{itemize}
\item[{(j)}] \ $\alpha+\beta+i \leq 2d+4 $
\item[{(k)}] \ $s-\beta+2 < 2^{2d-\alpha-\beta-i+5}$
\end{itemize} required by the above-mentioned lemma are deduced
directly from (a)---as with~(\ref{sacala}), 
for (k) it is convenient to keep in mind
the simple inequality $2^m \geq 2m$ for every 
$m\in\mathbb{Z}$. Thus, we will assume in addition
\begin{equation}\label{F28}
\alpha+s < 2d+3.
\end{equation} 

Now, when $\alpha+s \leq 2d-2i+6$, the triple $(i-1,\alpha,s)$
satisfies the conditions (a), (b), and (c) of
Proposition~\ref{P2}. Then, the case (d) of the induction assures
that (\ref{F25}) is a consequence of the inequality $\nu\binom{2^e}{
s -\beta+2} \geq \beta+e-s$ which, just as in the considerations previous
to (\ref{F27}), follows directly from Lemma~\ref{L1}. Therefore, we
can now assume 
\begin{equation}\label{F29}
2d-2i+6 < \alpha+s.
\end{equation}
Note that (\ref{F28}) and (\ref{F29}) imply that (\ref
{F27}) specializes to $i \geq 3$,
while (\ref{F29}) and (a) imply that (\ref{F26})
specializes to
\begin{equation}\label{F30}
s > \beta+2.
\end{equation}

At this point we are in conditions to repeat the analysis
performed after (\ref{F28}): when $\alpha+s \leq 2d-2i+8$, and by
virtue of (\ref{F29}), the triple $(i-2,\alpha,s)$ satisfies the
conditions (a), (b), and (c) of Proposition~\ref{P2}, so the
clause (d) of the induction assures that (\ref{F25}) is a
consequence of the inequality $\nu\binom{2^e}{s-\beta+2} \geq
\beta+e-s+1$, which is guaranteed by Lemma~\ref{L1}---the
verification of the second hypothesis in such lemma uses (\ref{F30})
together with the fact that $2^m> m+2$, for $m> 2$. As a result, we can
assume that~(\ref{F29}) specializes to
\begin{eqnarray}
2d-2i+8 < \alpha+s. \nonumber
\end{eqnarray} 

Recursively, assume that, for some $j \geq 3$,
(\ref{F25}) has been proved except for
\begin{equation}\label{F31}
2(d-i+j+1) < \alpha+s.
\end{equation}
Note that (\ref{F28}) and (\ref{F31}) imply that $i\geq
j+1$; whereas (\ref{F31}) and (a) imply
\begin{equation}\label{F32}
s > \beta+2(j-1).
\end{equation}
Then, if (\ref{F31}) is satisfied together with the
inequality $\alpha+s \leq 2 (d-i+j+2)$, the triple $(i-j, \alpha,
s)$ fulfills the conditions (a), (b), and (c) in
Proposition~\ref{P2}. Therefore, the modality (d) of the induction
assures that (\ref{F25}) is a consequence of the inequality $\nu
\binom{2^e}{s-\beta+2} \geq \beta+e-s+j-1$, which in turn is
guaranteed by Lemma~\ref{L1} (the verification of the second
hypothesis in the aforementioned lemma uses (\ref{F32}) together with the
fact that $2^m> m+j$, for $m>j$). This closes the recursive
process (which is finite, in view of (\ref {F28})), concluding the
verification of (\ref{F25}) and, thus, the proof of Proposition~\ref{P2}.
\end{proof}

We are now ready to prove the
main result of this section, Conjecture~\ref{C} for $j \leq
2$. This result will be used in the proof of Proposition~\ref{xnyn}, which
is a key step in the proof of Theorem~\ref{CTC} and
Corollary~\ref{cotainferior}, our main applications to the
topological complexity of lens spaces.

\begin{proof}[Proof of Theorem~\emph{\ref{contribucion}}] 
The easy part is for $j=1$, for which we will see that all contributions
to the term $v_1^d(0,0)$ on the right-hand side of (\ref{F4}) have
integer coefficients divisible by $2^e$. The contributions that come
from the second summation arise from indexes of the form $(r,0)$
with $r\geq0$ and the corresponding coefficient is divisible by $\binom{2^k}{
r+1}$. But $\nu\binom{2^k}{r+1} \geq e$ in view of Lemma~\ref{L1}.
Indeed, the relation $r+1 <2^{k-e+1} $ is assured by (\ref{F7}),
that in our case is translated into $r\leq d=k-e$. On the other
hand, the contributions that come from the first summation
in~(\ref{F4}) arise from indexes of the form $(0,s)$ with $s\geq 0$
and they have coefficient $\binom{2^e}{s+1} c_{0, s}$. In such
cases, making use of Proposition~\ref{P1}, it suffices to
verify the inequality $\nu\binom{2^e}{s+1} \geq e-s$. But the latter
relation is an easy consequence of Lemma~\ref{L1}.

\medskip
The crux of the matter is dealing with
the case $j=2$. For organizational purposes we start by settling 
a few preliminary steps. 

\medskip\noindent{\bf Step 1}. Let us focus attention on the
multiples of a basis element of the form $(\alpha,\beta)$ with
$\alpha +\beta=2d+3$. As we have noticed before, 
such a basis element arises on the right hand side of~(\ref{F4}) only
from the summand with index $(\alpha,\beta)$, in both summations of
(\ref{F4}). Thus, we obtain the relation
\begin{eqnarray}\label{FA}
0=2^ec_{\alpha,\beta}+2^kd_{\alpha,\beta}.
\end{eqnarray} But in view of Proposition~\ref{P1}, the coefficient
$c_{\alpha,\beta}$ in (\ref{FA}) takes the form $c_{\alpha,\beta} =
2^dc'_{\alpha,\beta}$ for some integer $c'_{\alpha,\beta}$. In these
terms, the only information we need
from (\ref{FA}) is given by the mod-2 congruences
\begin{equation}\label{FB}
\begin{array}{rcl}
c'_{1,2d+2}+ d_{1,2d+2} & \equiv & 0,\\
c'_{2,2d+1}+ d_{2,2d+1} & \equiv & 0.
\end{array}
\end{equation}

\medskip\noindent{\bf Step 2}. Let us now focus attention on the way 
a basis element $(\alpha,\beta)$ with $\alpha+\beta=2d+2$ arises
on the right hand side of~(\ref{F4}).
The total coefficient of (the $v_1$-multiple
of) $(\alpha, \beta)$ is
$$
0=2^ec_{\alpha,\beta}+ \binom{2^e}{2}c_{\alpha,\beta+1}+2^kd_{\alpha,\beta}
+\binom{2^k}{2}d_{\alpha+1,\beta}.
$$
The first summand on the right hand side
of this expression vanishes modulo $2^k$ provided $0\leq\alpha\leq2$ because,
if $e <k$, Proposition~\ref{P2} (with the triple $(1,\alpha,\beta)$)
produces $\nu(c_{\alpha,\beta}) \geq d$. We thus have the congruences modulo $2^k$
\begin{eqnarray}\label{FD}
0\equiv \binom{2^e}{2}c_{\alpha,\beta+1}+ \binom{2^k}{
2}d_{\alpha+1,\beta},\quad\mbox{for $0\leq\alpha\leq2$}.
\end{eqnarray} Making use of the fact that $\nu\binom{2^q}{2}=q-1$,
together with the expression $c_{\alpha,\beta+1}=
2^dc'_{\alpha,\beta+1}$ described above, we see that (\ref{FD}) is
equivalent to the following three congruences modulo 2:
\begin{eqnarray}\label{FZ}
0\equiv c'_{0,2d+3}+ d_{1,2d+2},\nonumber\\
0\equiv c'_{1,2d+2}+ d_{2,2d+1},\nonumber\\
0\equiv c'_{2,2d+1}+ d_{3,2d}.\quad\, \nonumber
\end{eqnarray} Together with the congruences in (\ref{FB}), this leads
to the mod~2 congruence
\begin{eqnarray}\label{FE}
0\equiv c'_{0,2d+3}+ d_{3,2d}.
\end{eqnarray} 

\medskip\noindent{\bf Step 3 (conclusion of the proof)}. 
Let us consider the way in which the basis element $(0,2d)$ arises on the right hand 
side of~(\ref{F4}). Notice that this basis element appears on the left hand
side of (\ref{F4}) only when $e=k$. So, the total coefficient
with which (the $v_1^3$-multiple of) $(0,2d)$ appears in (\ref{F4}) is
\begin{eqnarray}\label{FF}
\,\hspace{8mm}2^{e-2}\delta(e,k)=2^ec_{0,2d}+ \binom{2^e}{2}c_{0,2d+1}+
\binom{2^e}{3}c_{0,2d+2} + \binom{2^e}{4}c_{0,2d+3}\nonumber \\
\quad +{\,} 2^kd_{0,2d}+\binom{2^k}{2}d_{1,2d}+\binom{2^k}{
3}d_{2,2d}+\binom{2^k}{4}d_{3,2d}
\end{eqnarray}
where $\delta(e,k)$ is Kronecker's delta.
We next show that each of the terms on the right hand side of
(\ref{FF}) vanishes modulo $2^{k-1}$, except perhaps for the
second, fourth, and last terms.

\medskip
\noindent {\bf Case of $2^ec_{0,2d}$}. The claim is obvious for
$e+1\geq k$. When $k> e+1$, Proposition~\ref{P2} (with the triple
$(i,\alpha,\beta) = (2,0,2d)$) produces $\nu(c_{0,2d}) \geq d-1$, so
that $\nu(2^ec_{0,2d}) \geq e+d-1=k-1$.

\smallskip
\noindent {\bf Case of $\binom{2^e}{3}c_{0,2d+2}$}. On one hand, we have
$\nu\binom{2^e}{3}=e$, and on the other, Proposition~\ref{P1}
implies that $\nu(c_{0,2d+2})\geq d-1$, so that $\nu(\binom{2^e}{3}
c_{0,2d+2}) \geq e+d-1=k-1$.

\smallskip
\noindent {\bf Case of the fifth, sixth, and seventh terms on the
right hand side of (\ref{FF})}. The affirmation is obvious since
$\nu\binom{2^k}{2}=k-1$ and $\nu\binom{2^k}{3}=k$.

\medskip
Therefore, modulo $2^{k-1}$, (\ref{FF}) reduces to
\begin{eqnarray}\label{FG}
2^{e-2}\delta(e,k)\equiv \binom{2^e}{2}c_{0,2d+1} + \binom{2^e}{
4}c_{0,2d+3}+\binom{2^k}{4}d_{3,2d}.
\end{eqnarray} 
Now, when $e=k$, the first term on the right
hand side of
(\ref{FG}) is trivial, whereas the rest of the binomial coefficients
are exactly divisible by $2^{k-2}$, so that (\ref{FG}) simplifies
to the congruence
\begin{eqnarray}\label{FY}
1\equiv c_{0,2d+3} + d_{3,2d}\nonumber
\end{eqnarray} modulo 2, which contradicts (\ref{FE}) and, thus, 
completes the proof of
Theorem~\ref{contribucion} for $e=k$. On the other hand,
when $k> e$, equation~(\ref{FG}) takes the form
\begin{eqnarray}\label{FH}
0\equiv \binom{2^e}{2}c_{0,2d+1} + \binom{2^e}{
4}c_{0,2d+3}+\binom{2^k}{4}d_{3,2d}
\end{eqnarray} modulo $2^{k-1}$. But in terms of the notation
$c_{0,2d+3} = 2^{d} c'_{0,2d+3}$ and the corresponding fact that
$c_{0,2d+1} = 2^{d-1} c'_{0,2d+1}$ with
$c'_{0,2d+1}$ an odd integer (justified by Lemma~\ref{C1} below),~(\ref{FH}) 
translates as
$0\equiv c'_{0,2d+1} + c'_{0,2d+3}+d_{3,2d}\equiv 1 +
c'_{0,2d+3}+d_{3,2d}$ modulo 2 which, again, contradicts (\ref{FE}).
\end{proof}

It only remains to prove that $c_{0,2d+1} = 2^{d-1} c'_{0,2d+1}$ with
$c'_{0,2d+1}$ an odd integer. 
Note that the inequality $\nu(c_{0,2d+1}) \geq d-1$ is
guaranteed for $d\geq 2$ by Proposition~\ref{P2} with 
$(i,\alpha,\beta) = (1,0,2d+1)$. Then, our aim is  to prove
that such inequality can be refined to an equality for $d\geq1$.
We prove in fact:
\begin{lema}\label{C1}
In the case $j=2$ of~\emph{(\ref{F4})}, we have
$\nu(c_{0,2\ell+1})=\ell-1$ for $\ell=1,2,\dots,d$.
\end{lema}
\begin{proof} We proceed by induction on $\ell$. For $\ell=1$, we start by
analyzing the coefficients of $v_1^{2d+3}(0,0)$ on both
sides of~(\ref{F4}). On the left hand side, the
coefficient is  $2^{e-2}$. On the right hand side, the
coefficients coming from the second summation 
arise from summands having index $(\alpha,0)$ with 
\begin{equation}\label{siempresi}
\alpha +1 \leq 2^k \mbox{ \ and \ } \alpha \leq 2d+3
\,\mbox{{}---the latter in view of (\ref{F7}). }
\end{equation}
The relevant coefficient, $\binom{2^k}{\alpha+1}$,
is divisible by $2^{e-1}$ in view 
of Lemma~\ref{L1} since the required hypotheses
$e-1 \leq k$ and $\alpha +1 < 2^{k-e+2}$ are satisfied---the
latter follows from the second inequality in~(\ref{siempresi}) and
the facts that $d\geq1$ and
\begin{equation}\label{F39}
2m < 2^m \quad \textrm{for all} \quad m \geq 3.
\end{equation}
On the other hand, the coefficients coming from the first
summation of (\ref{F4}) arise from summands having an index $(0,\beta)$
with
\begin{eqnarray}\label{F41}
\beta \leq 2d+3,\, \mbox{ \ \ and \ \ } \, \beta
+1 \leq 2^e.
\end{eqnarray}
For such summands, the relevant coefficient is
$c_{0,\beta}\binom{2^e}{\beta +1}$. We show in the next paragraph
\begin{equation}\label{F42}
\nu (c_{0,\beta})+ \nu\binom{2^e}{\beta+1} \geq e-1
\end{equation} for $\beta \neq 3$. Once this is done,
the remaining coefficient to analyze---the one with $\beta=3$---will be 
forced to be divisible by  $2^{e-2}$ but
not by $2^{e-1}$, so that $\nu(c_{0,3}) + \nu\binom{2^{e}}{4} =
e-2$, that is $\nu(c_{0,3}) =0$, completing the proof 
for $\ell=1$.

\medskip When $0 \leq \beta \leq 2$ we have
\begin{eqnarray*}
\nu\binom{2^{e}}{\beta + 1} =  \left\{\begin{array}{ll}
\nu\binom{2^{e}}{1}= e, & \textrm{if}\quad \beta = 0,\vspace{2mm}\\
\nu\binom{2^{e}}{2}= e-1, & \textrm{if}\quad \beta = 1,\vspace{2mm}\\
\nu\binom{2^{e}}{3}= e, & \textrm{if}\quad \beta = 2,
\end{array} \right.
\end{eqnarray*}
implying~(\ref{F42}). So, we can safely assume 
$\beta \geq 4$. If $\beta$ is even,~(\ref{F42})
is evident as $\nu \binom{2^{e}}{
\beta + 1} = e$. If $\beta = 2m+1$, then $m\geq 2$ and~(\ref{F41}) reads
\begin{eqnarray}
2m + 1 \leq 2d+3, \quad\textrm{i.e.} \quad m\leq d+1,\nonumber
\label{F43}
\end{eqnarray}\vspace{-0.65cm}
\begin{eqnarray}
2m+2 \leq 2^e, \quad\textrm{i.e.} \quad m < 2^{e-1}. \nonumber
\label{F44}
\end{eqnarray} We then consider two cases:

\medskip
\noindent {\bf Case $m=d+1$}. In view of Proposition~\ref{P1}, it suffices to verify the inequality $$2m+1-d-3+ \nu\binom{2^{e}}{
\beta + 1}\geq e-1,$$
that is
\begin{eqnarray}\label{F45}
\nu\binom{2^{e}}{\beta + 1}\geq e-d-1.\nonumber
\end{eqnarray}
But this is a consequence of Lemma~\ref{L1} since the
two required hypotheses $e-d- 1 \leq e$ and $\beta + 1 < 2^{d+2}$ 
hold---the latter follows from
(\ref{F39}) and the fact that $\beta +1 = 2d + 4$.

\medskip
\noindent {\bf Case $m\leq d$}. The three hypotheses in Proposition~\ref{P2}
for $i=d+1-m$ and $c_{0,\beta}$ are

\vspace{-5mm}
\begin{eqnarray*}
&d+4-d-1+m \leq \beta \leq 2d-2d-2+2m+4,&\\
&0\leq d+3-d-1+m,&\\
&1\leq d+1-m \leq d,&
\end{eqnarray*}
or, equivalently
\begin{eqnarray*}
&m+3 \leq \beta \leq 2m+2,&\\
&0\leq 2+m,&\\
&1\leq m\leq d,&
\end{eqnarray*}
all of which are obvious. Therefore it suffices to check that $$\beta
+d+1-m-d-3+\nu\binom{2^{e}}{\beta + 1}\geq e-1.$$ But since
$\beta=2m+1$, this simplifies to $\nu\binom{2^{e}}{\beta + 1}
\geq e-m$ which, in turn, 
follows from Lemma~\ref{L1} as the two required hypotheses
$e-m\leq e$ and $\beta+1 <2^{m+1}$ hold---the latter 
follows from (\ref{F39}) and $\beta=2m+1$. This completes the verification 
of~(\ref{F42}) and, thus, of the case $\ell=1$ of the lemma.

\medskip
For the inductive step of the proof ($\ell\geq2)$,
we analyze the coefficients for $v_1^{2d-2\ell+5} (0,2\ell-2)$ on both 
sides of~(\ref{F4}). The coefficient is zero on the left hand side. On
the right hand side, the coefficients coming from the second summation arise 
from summands having index $(\alpha,2\ell-2)$ with
\begin{eqnarray}
\alpha+1\leq 2^k \,\mbox{ \ \ and \ \ }\,
2\ell+\alpha -2 \leq 2d+3. \label{F56}
\end{eqnarray}
In view of Lemma~\ref{L1}, 
the coefficient for such a summand, $\binom{2^{k}}{\alpha + 1}$, is 
divisible by $2^{e+\ell-2}$
as $k \geq k-2=e+d-2\geq e+\ell-2$, $\,\alpha+1 \leq 2d+6-2\ell$ 
by~(\ref{F56}), and $2d+6-2\ell < 2^{d-\ell+3}
= 2^{k-e-\ell+3}$ by~(\ref{F39}). 
On the other hand, the coefficients coming from the first
summation in~(\ref{F4}) arise from summands having an index $(0,\beta)$
with
$$
2\ell - 2 \leq \beta, \quad \beta \leq 2d+3, \,\mbox{
\ \ and \ \ }\,\beta-2\ell+3\leq 2^e. \label{F59}
$$
We claim that the coefficient $c_{0,\beta}\binom{2^{e}}{
\beta-2\ell + 3}$ of such a summand is also divisible by $2^{e+\ell-2}$
provided $\beta\neq 2\ell+1, 2\ell-1$. Indeed:

\medskip
\noindent {\bf Case $\beta=2d+3$}. In view of Proposition~\ref{P1}, it
suffices to verify $$\beta-d-3 + \nu\binom{2^{e}}{
\beta-2\ell + 3} \geq e+\ell-2,$$ that is
$$
\nu\binom{2^{e}}{\beta -2\ell + 3} \geq e+\ell-d-2,
$$
which follows from Lemma~\ref{L1} as the two required hypotheses 
$\ell\leq d+2$  and $\beta-2\ell+3 = 2d-2\ell+6 < 2^{d+3-\ell}$ hold---the latter 
by virtue of (\ref{F39}) and since
$d+3-\ell \geq 3$.

\medskip\noindent {\bf Case $\beta=2\ell-2$}. We need to verify
$$\nu(c_{0,2\ell-2})+ \nu\binom{2^{e}}{1}\geq e+\ell-2,$$ that is
\begin{equation}\label{F62}
\nu(c_{0,2\ell-2})\geq\ell-2,
\end{equation}
and we can safely assume $\ell\geq 3$. The hypotheses of
Proposition~\ref{P2} for $c_{0,2\ell-2}$ with $i=d-\ell+3$ read as
\begin{eqnarray*}
&d+4-(d+3-\ell)\leq 2\ell-2 \leq 2d-2(d+3-\ell)+4,&\\
&0 \leq d+3-d+\ell-3,&\\
&1 \leq d-\ell+3 \leq d,&
\end{eqnarray*}
which are equivalent to
\begin{eqnarray*}
&1+\ell\leq 2\ell-2 \leq 2\ell-2,&\\
&0 \leq \ell,&\\
&3 \leq \ell\leq d+2,&
\end{eqnarray*} 
respectively, all of which are clearly satisfied.
Therefore~(\ref{F62}) follows from Proposition~\ref{P2}
in this case.

\medskip
\noindent {\bf Case $\beta=2\ell$}. We now need to verify
$\nu(c_{0,2\ell})+ \nu\binom{2^{e}}{3}\geq e+\ell-2,$ or alternatively
\begin{equation}\label{F69}
\nu(c_{0,2\ell})\geq \ell-2,
\end{equation}
which follows from Proposition~\ref{P2} for $c_{0,2\ell}$ with $i=d+1-\ell$.
This time the required hypotheses are
\begin{eqnarray*}
&d+4-(d+1-\ell)\leq 2\ell \leq 2d-2(d+1-\ell)+4,&\\
&0\leq d+3-(d+1-\ell),&\\
&1\leq d+1-\ell \leq d,&
\end{eqnarray*}
or equivalently
\begin{eqnarray*}
&3\leq \ell \quad \textrm{and} \quad 2\ell\leq 2\ell+2,&\\
&0 \leq \ell+2,&\\
&1 \leq \ell\leq d,&
\end{eqnarray*}
respectively, which clearly hold (except for
$\ell=2$, in which case (\ref{F69}) is obvious).

\medskip
\noindent {\bf Case where $\beta$ satisfies}
\begin{eqnarray}
2\ell +2\leq \beta \leq 2d+2 \quad \textrm{\bf and} \quad \beta -2\ell+3 \leq
2^e. \label{F76}
\end{eqnarray}
We have to verify the inequality
\begin{equation}\label{verificar}\nu(c_{0,\beta})+
\nu\binom{2^{e}}{\beta -2\ell + 3}\geq e+\ell-2.
\end{equation} 
Let $[r]$ denote the integral part of a real number $r$.
We claim that the three hypotheses
of Proposition~\ref{P2} for $c_{0,\beta}$ with 
$i=d+2-\left[\frac{\beta+1}{2}\right]$ hold, namely:
\begin{eqnarray}
&&d+4-d-2+\left[\frac{\beta+1}{2}\right]\leq \beta \leq
2d-2d-4+2\left[\frac{\beta+1}{2}\right]+4, \label{F78}\\
&&0 \leq d+3-d-2+\left[\frac{\beta+1}{2}\right], \label{F79}\\
&&1 \leq d+2-\left[\frac{\beta+1}{2}\right] \leq d. \label{F80}
\end{eqnarray}
The inequality~(\ref{F79}) is clear. The inequalities in
(\ref{F78}) are verified by a straightforward argument: If $\beta=2m$, then
(\ref{F78}) is equivalent to $2+m\leq 2m\leq 2m$, which is
true since, by (\ref {F76}), $m\geq 2$. If $\beta=2m+1$, then
(\ref {F78}) is equivalent to $2+m+1 \leq 2m+1\leq 2m+2$, which
again is valid because $m\geq 2$. Lastly, the second inequality in~(\ref{F80}) 
is evident, while the first one follows from~(\ref{F76}):
$$\left[\frac{\beta+1}{2}\right] \leq
\left[\frac{2d+3}{2}\right]=d+1.$$ Thus, the conclusion
in Proposition~\ref{P2} allows us to deduce~(\ref{verificar}) from
$$\beta+d+2- \left[\frac{\beta+1}{2}\right]-d-3+\nu\binom{2^{e}}{\beta
-2\ell + 3}\geq e+\ell-2,$$ that is
\begin{equation}\label{F81}
\nu\binom{2^{e}}{\beta -2\ell + 3}\geq e+\ell-\beta
+\left[\frac{\beta+1}{2}\right]-1.
\end{equation}
To prove (\ref{F81}), we check that the hypotheses in
Lemma~\ref{L1} are satisfied. The first one is $$e+\ell-\beta +
\left[\frac {\beta+1} {2} \right]-1\leq e,$$ or equivalently $$\ell +
\left[\frac {\beta+1} {2} \right] \leq \beta+1,$$ which is fulfilled
since, by (\ref {F76}), we have $\ell+1 
\leq \left[\frac{\beta}{2}\right]$, so that
$$\ell + \left[\frac {\beta+1} {2} \right] \leq \left[\frac{\beta}{2}
\right]-1 +\left[\frac {\beta+1} {2} \right] \leq \beta+1$$ 
---the last inequality can be verified in a simple direct way.
The second hypotheses in Lemma~\ref{L1} in
order to verify~(\ref{F81}) is that
\begin{equation}\label{F82}
\beta - 2\ell+3 < 2^{\beta+2-\ell- \left[\frac{\beta+1}{2}\right]}.
\end{equation}
If $\beta=2m+1$, (\ref{F82}) is equivalent to
$2m+1 - 2\ell+3 < 2^{2m+1-\ell-m-1+2}=2^{m-\ell+2}$,
which is true by virtue of (\ref{F39}), since $m-\ell+2\geq 3$ (by
(\ref{F76})). If $\beta=2m$, then (\ref{F82}) is equivalent to
$2m - 2\ell+3 < 2^{2m-\ell-m+2}=2^{m-\ell+2},$
which is true for the same reasons as in the case of an odd $\beta$.

\medskip
Therefore, except for the two cases not yet analyzed, namely
the terms in the first summation in~(\ref{F4}) with index 
$(0,2\ell-1)$ and $(0,2\ell+1)$, all the coefficients 
of contributions to the term $v_1^{2d-2\ell+5}(0,2\ell-2)$ have been
verified to be 
divisible by $2^{e+\ell-2}$. We consider now the case $\beta=2\ell-1$. The
greatest power of 2 that divides the corresponding contribution is
given by
$$
\nu(c_{0,2\ell-1}) + \nu\binom{2^{e}}{2} =  \nu(c_{0,2(\ell-1)+1})
+ e-1 = \ell-2+e-1=e+\ell-3,
$$
where the penultimate equality holds by induction.
Consequently, the coefficient coming from the case $\beta=2\ell+1$ must be
divisible by $2^{e+\ell-3}$, but not by $2^{e+\ell-2}$; that is
$$
e+\ell-3 = \nu(c_{0,2\ell+1}) + \nu\binom{2^{e}}{4}
=\nu(c_{0,2\ell+1})+e-2.\nonumber
$$
Thus $\nu(c_{0,2\ell+1}) = \ell-1$, completing the induction.
\end{proof}

\section{Proof of Theorem~\ref{CTC}}\label{lasecnueva}
The proof of Theorem~\ref{CTC} is based on Propositions~\ref{biequi} 
and~\ref{xnyn} below, the first
of which follows from Proposition~2.7 and Note 2.6
in~\cite{G}. The proof of Proposition~\ref{xnyn}  will be given 
at the end of this section,
after having deduced Theorem~\ref{CTC}.

\begin{proposition}\label{biequi}
There is a $\mathbb{Z}_{2^{e}}$-biequivariant map
$S^{2m+1}\times S^{2m+1}\to S^{2\ell+1}$ if and only if
there is a map $\beta:L^{2m+1}(2^{e})\times
L^{2m+1}(2^{e})\to L^{2\ell+1}(2^{e})$
such that the diagram

\vspace{-5mm}
$$
\xymatrix{ L^{\infty}(2^{e})\times L^{\infty}(2^{e}) \ar[r]^-\mu
& L^{\infty}(2^{e})  \\
L^{2m+1}(2^{e})\times L^{2m+1}(2^{e}) \ar[u]
\ar[r]^(0.66){\beta\;\,} & L^{2\ell+1}(2^{e})\ar[u] }
$$
is homotopy commutative, where
$\mu$ is the $H$-space product on
$L^{\infty}(2^{e})=K(\mathbb{Z}/2^{e},1)$, and both vertical maps are inclusions.
\end{proposition}

\begin{proposition}\label{xnyn}
If $\alpha(n)\ge3$, the element $2^{\alpha(n)-3}v_1^{3}x^{n+1}y^{n+1}$ is non-zero
in the ring $R_n=ku_*[x,y]/([2^{\alpha(n)-1}]
(x),[2^{\alpha(n)-1}](y),x^{n+2},y^{n+2})$.
\end{proposition}

\begin{proof}[Proof of Theorem~\ref{CTC}]
We assume $b(n+1,\alpha(n)-1)<2n-2$ and derive a 
contradiction. Since the case $\alpha(n)=2$ follows from~(\ref{cota}),
we will assume throughout the section that $\alpha(n)\geq3$. By
Proposition~\ref{biequi}, the restriction to 
$L^{2n+3}(2^{\alpha(n)-1})\times L^{2n+3}(2^{\alpha(n)-1})$ of the
$H$-product in $L^\infty(2^{\alpha(n)-1})$ is homotopic to a map 
$\beta\colon L^{2n+3}(2^{\alpha(n)-1})\times L^{2n+3}(2^{\alpha(n)-1})\to
L^{4n-5}(2^{\alpha(n)-1})$. Apply $ku$-cohomology to 
the homotopy commutative diagram

\vspace{-4mm}
$$\xymatrix{
& & \mbox{\footnotesize $L^{2n+2}(2^{\alpha(n)-1}){\times} L^{2n+2}
(2^{\alpha(n)-1})$} \ar[dl]\ar[d]\\
& \mbox{\footnotesize $L^{\infty}(2^{\alpha(n)-1}){\times} L^{\infty}
(2^{\alpha(n)-1})$} \ar[dl]_{\pi\times\pi} \ar[d]^{\mu} &
\mbox{\footnotesize $\,L^{2n+3}(2^{\alpha(n)-1}){\times}
L^{2n+3}(2^{\alpha(n)-1})$}
 \ar[d]^{\beta} \ar[l]\\
\mbox{\footnotesize $\mathbb{C}P^{\infty}{\times} \mathbb{C}
P^{\infty}$}\ar[d]^{\mu'} & \mbox{\footnotesize $L^{\infty}
(2^{\alpha(n)-1})$}\ar[dl]^{\pi} & \mbox{\footnotesize $\,L^{4n-5}(
2^{\alpha(n)-1})$} \ar[l] \ar[dl]^{\pi}\\
\mbox{\footnotesize $\mathbb{C}P^{\infty}$} & \mbox{\footnotesize
$\,\,\mathbb{C}P^{2n-3}$}\ar[l] }$$

\noindent where unlabeled maps are inclusions, $\mu'$ is the
$H$-product of $\mathbb{C}P^{\infty}$, and $\pi\colon L^{2k+1}(
2^{\alpha(n)-1})\to\mathbb{C}P^k$ ($0\leq
k\leq\infty$) stands for the canonical projection (a morphism of
$H$-spaces for $k=\infty$). 

It is well known that the maps $\pi^*$ and $(\pi\times\pi)^*$ 
induced in $ku$-theory take respectively the forms
$ku_*[[x]]\to ku_*[[x]]/([2^{\alpha(n)-1}](x))$ and 
$ku_*[[x,y]]\to ku_*[[x,y]]/([2^{\alpha(n)-1}](x),
[2^{\alpha(n)-1}](y))$, where $x$ and $y$ stand for 
complex orientation elements of cohomological dimension 2 and, as usual
in this context, $ku_*$ is graded over the non-positive integers. 
Further, in even dimensions the ring $ku^{*}(L^{2n+2}(2^{\alpha(n)-1}) \times
L^{2n+2}(2^{\alpha(n)-1}))$ agrees with $R_n$, and 
the ``restriction'' map $ku^{\mathrm{even}}(L^{\infty}(2^{\alpha(n)-1}) \times
L^{\infty}(2^{\alpha(n)-1})) \to R_n$
is surjective with kernel generated by
$x^{n+2}$ and $y^{n+2}$ (cf.~Proposition~3.1 
in~\cite{GN}). Chase the element $x\in
ku^{*}(\mathbb{C}P^{\infty})$ across the diagram above, from
the lower left corner to the upper right corner, recalling that 
$(\mu')^{*}(x) =x+y+v_1xy$ is the multiplicative formal group.
Since $x\in ku^{*}(\mathbb{C}P^{2n-3})$ satisfies $x^{2n-2}=0$, 
we get the relation $(x+y+v_1xy)^{2n-2}=0$ in $R_n$. The proof will be
complete once we show that, up to units 
in this ring, $x(x+y+v_1xy)^{2n-2}=2^{\alpha(n)-3}
v_1^3x^{n+1}y^{n+1}$, the element in Proposition~\ref{xnyn}.

\smallskip
In view of the relations $x^{n+2}=y^{n+2}=0$, we have
\begin{eqnarray*}
\lefteqn{x(x+y+v_1xy)^{2n-2} = x(x+y)^{2n-2}+ 
\mbox{$\binom{2n-2}{1}$}(x+y)^{2n-3}v_1x^2y}\hspace{1cm}\;\\
& & +\mbox{$\binom{2n-2}{2}$}(x+y)^{2n-4}v_1^2x^3y^2
+\mbox{$\binom{2n-2}{3}$}(x+y)^{2n-5}v_1^3x^4y^3.
\end{eqnarray*}
These four summands reduce, respectively, to
\begin{eqnarray}
&&\;\;\;\mbox{$\binom{2n-2}{n}$}x^{n+1}y^{n-2} \hspace{.3mm}{+}\hspace{.3mm} 
\mbox{$\binom{2n-2}{n-1}$}
x^{n}y^{n-1} \hspace{.3mm}{+}\hspace{.3mm}\mbox{$\binom{2n-2}{n-2}$}x^{n-1}y^{n}
\hspace{.3mm}{+}\hspace{.3mm}
\mbox{$\binom{2n-2}{n-3}$}x^{n-2}y^{n+1};\label{cbin1}
\\&&\rule{0mm}{5mm}\;\;\;\mbox{$\binom{2n-2}{1}$}
\left[\mbox{$\binom{2n-3}{n-1}$}
x^{n-1}y^{n-2} + \mbox{$\binom{2n-3}{n-2}$}x^{n-2}y^{n-1} +
\mbox{$\binom{2n-3}{n-3}$}x^{n-3}y^{n}\right]v_1x^2y;\label{cbin2}
\\&&\rule{0mm}{5mm}\;\;\;\mbox{$\binom{2n-2}{2}$}\left[
\mbox{$\binom{2n-4}{n-2}$}x^{n-2}y^{n-2} +
\mbox{$\binom{2n-4}{n-3}$}x^{n-3}y^{n-1}\right]v_1^2x^3y^2;\label{cbin3}
\\&&\rule{0mm}{5mm}\;\;\;\mbox{$\binom{2n-2}{3}$}\left[
\mbox{$\binom{2n-5}{n-3}$}x^{n-3}y^{n-2}\right]v_1^3x^4y^3.\label{cbin4}
\end{eqnarray}
Next we analyze the divisibility by 2 of
these binomial coefficients assuming
\begin{eqnarray}\label{dimalfa}
\alpha(n+1)\geq \alpha(n)
\end{eqnarray}
(Theorem~\ref{CTC} follows from~(\ref{topcases}) for $\alpha(n)>\alpha(n+1)$).
In the following arguments we make free use of
the relations $\nu\binom{a}{b}=\alpha(b)+\alpha(a-b)-\alpha(a)$,
$\,\alpha(k-1)=\alpha(k)-1+\nu(k),\,$ and $\,2^iz^j=0$ for 
$i+j\geq n+\alpha(n)$ and
$z\in\{x,y\}$ (the latter was pointed out in~(\ref{ord})
to come from Corollary~2.6 in~\cite{cfimm}).

\begin{lema}\label{pbin1}
The first and last terms in~\emph{(\ref{cbin1})} vanish.
\end{lema}
\begin{proof}\ 

\vspace{-1.1cm}\begin{eqnarray*}
\nu\mbox{$\binom{2n-2}{n}$} & = &\alpha(n)+\alpha(n-2)- \alpha(n-1)\\
& = & \alpha(n)+ \alpha(n-1)-1+ \nu(n-1)- \alpha(n-1)\\
& = & \alpha(n)-1+ \nu(n-1)\geq \alpha(n)-1,
\end{eqnarray*} 
so that in fact $\binom{2n-2}{n}x^{n+1}=0$, which takes care of the 
first summand in~(\ref{cbin1}). 
Similarly, the last summand in~(\ref{cbin1}) is taken care of by
observing that $\nu\binom{2n-2}{n-3}\geq\alpha(n+1)-2+\nu(n-1)+\nu(n-2)
\geq\alpha(n)-1$, where the last inequality uses~(\ref{dimalfa}). 
\end{proof}

\noindent The rest of the coefficients in~(\ref{cbin1}) are
analyzed in the following result.

\begin{lema}\label{pbin2} Up to units in $R_n$ we have
\begin{itemize}
\item[]$\rule{0mm}{4mm}\,\hspace{-8.7mm}\bullet\;\binom{2n-2}{n-1}
x^{n}y^{n-1}=2^{\alpha(n)-3}v_1^3x^{n+1}y^{n+1}$
and $\binom{2n-2}{n-2}x^{n-1}y^{n}=0$, for $n$ odd;
\item[]$\rule{0mm}{4mm}\,\hspace{-.1mm}\hspace{-8mm}
\bullet\hspace{.6mm}\;\binom{2n-2}{n-1}
x^{n}y^{n-1}=0$ and $\binom{2n-2}{n-2}
x^{n-1}y^{n}=2^{\alpha(n)-3}v_1^3x^{n+1}y^{n+1}$, for $n$ even.
\end{itemize}
\end{lema}
\begin{proof}We only consider the case with $n$ odd; the situation for $n$ even
is similar and, thus, left to the reader. With $n=2k+1$ we have
$$
\nu\mbox{$\binom{2n-2}{n-1}$} = \alpha(n-1)+ \alpha(n-1)- \alpha(n-1)
=  \alpha(n)-1+ \nu(2k+1)= \alpha(n)-1.
$$
Therefore, up to a 2-local unit, the second summand in~(\ref{cbin1}) is
$2^{\alpha(n)-1}x^{n}y^{n-1}$. Now, by using the $2^{\alpha(n)-1}$-series twice,
we get
\begin{eqnarray*}
\lefteqn{2^{\alpha(n)-1}x^{n}y^{n-1}=2^{\alpha(n)-1}xx^{n-1}y^{n-1}}\\
&= &x^{n-1}y^{n-1}\left(-\mbox{$\binom{2^{\alpha(n)-1}}{2}$}v_1x^2
-\mbox{$\binom{2^{\alpha(n)-1}}{3}$}v_1^2x^3-\mbox{$\binom{2^{\alpha(n)-1}}{4}$}
v_1^3x^4-\cdots\right)\\&= &-\mbox{$\binom{2^{\alpha(n)-1}}{2}$}
v_1x^{n+1}y^{n-1}=-\mbox{$\binom{2^{\alpha(n)-1}}{2}$}v_1y^2x^{n+1}y^{n-3}\\&= &
x^{n+1}y^{n-3}\left(2^{\alpha(n)-1}y+\mbox{$\binom{2^{\alpha(n)-1}}{3}$}v_1^2y^3+
\mbox{$\binom{2^{\alpha(n)-1}}{4}$}v_1^3y^4 +\cdots\right)\nonumber\\
&= & 2^{\alpha(n)-1}x^{n+1}y^{n-2}+\mbox{$\binom{2^{\alpha(n)-1}}{3}$}
v_1^2x^{n+1}y^{n}+\mbox{$\binom{2^{\alpha(n)-1}}{4}$}v_1^3x^{n+1}y^{n+1}.
\end{eqnarray*} 
The first two terms in the last expression vanish as 
$2^{\alpha(n)-1}x^{n+1}=0$, 
while the fact that $\nu\binom{2^{\alpha(n)-1}}{4}=
\alpha(n)-3$ means that, up to a 2-local unit, the third term takes the form 
$2^{\alpha(n)-3}\hspace{.2mm}v_1^3x^{n+1}y^{n+1}$. On the other hand
\begin{eqnarray*}
\nu\mbox{$\binom{2n-2}{n-2}$} & = &\alpha(n-2)+ \alpha(n)- \alpha(n-1)\\
& = & \alpha(n-1)-1+ \nu(n-1)+ \alpha(n)- \alpha(n-1)\\
& = & \alpha(n)-1+ \nu(2k) \geq \alpha(n),
\end{eqnarray*} 
and the third summand in~(\ref{cbin1}) vanishes as $2^{\alpha(n)}y^n=0$.
\end{proof}

The proof of Theorem~\ref{CTC} is complete by noticing that every summand in~(\ref{cbin2})--(\ref{cbin4}) vanishes (such a verification is similar to the previous arithmetical manipulations and, therefore, left to the reader).
\end{proof}

\begin{proof}[Proof of Proposition~\ref{xnyn}]
As a $ku_*$-module, the ring $R_n$ decomposes as the direct sum
of a free $ku_{*}$-module generated by the unit $1$, and a $ku_{*}$-module 
$M_n$ generated by the elements $x^{r}y^{s}$ with $0\leq r,s\leq n+1, r+s>0$
subject only to relations of the form
\begin{itemize}
\item $a_0x^{r+1}y^{s}+a_1x^{r+2}y^{s}+a_2x^{r+3}y^{s}+\cdots$
\item $a_0x^{r}y^{s+1}+a_1x^{r}y^{s+2}+a_2x^{r}y^{s+3}+\cdots$
\end{itemize}
(for $r$ and $s$ as above) where $a_i=\binom{2^{\alpha(n)-1}}{i+1}v_1^i$. In turn, $M_n$ is isomorphic to
the $ku_*$-module $M'_n$ generated by elements
$(i,j)$ with $0\leq i,j\leq n+1$ and $i+j<2n+2$
subject only to the relations
\begin{itemize}
\item $a_0(i-1,j)+a_1(i-2,j)+\cdots$
\item $a_0(i,j-1)+a_1(i,j-2)+\cdots$
\end{itemize}
(for $r$ and $s$ as above) An isomorphism $M_n\cong M'_n$ 
identifies $x^{r}y^{s}$ with 
$(n+1-r,n+1-s)$---note that $ku_*$ recovers its non-negative grading in the 
$ku_*$-module structure of $M'_n$. Now, we have an obvious map
$M'_n\to M''_n$ where $M''_n$ is generated by elements 
$(i,j)$ with $0\leq i,j\leq n+1$ subject only to relations of the form
\begin{itemize}
\item $a_0(i,j)+a_1(i-1,j)+a_2(i-2,j)+\cdots$
\item $a_0(i,j)+a_1(i,j-1)+a_2(i,j-2)+\cdots$
\end{itemize}
(for $r$ and $s$ as above). Note that, in $M''_n$,
the generator $(n+1,n+1)$ has been included, as well as
a slightly larger set of relations has been imposed, namely, besides the
relations in $M'_n$, $M''_n$ also has the relations starting
as $a_0(n+1,n+1)+a_1(n+1,n)+\cdots$ and
$a_0(n+1,n+1)+a_1(n,n+1)+\cdots$. Finally, $M''_n$ maps canonically to 
$ku_{*}(L^{\infty}(2^{\alpha(n)-1})) \otimes ku_{*}(L^{\infty}(2^{\alpha(n)-1}))$---where
$i$ and $j$ vary over all non-negative integers. Since the 
composition $M_n\cong M'_n\to M''_n\to 
ku_{*}(L^{\infty}(2^{\alpha(n)-1})) \otimes ku_{*}(L^{\infty}(2^{\alpha(n)-1}))$
sends $x^{n+1}y^{n+1}$ to $(0,0)$, Proposition~\ref{xnyn}
follows from Theorem~\ref{contribucion} since,
as an element of 
$ku_{*}(L^{\infty}(2^{\alpha(n)-1})) \otimes ku_{*}(L^{\infty}(2^{\alpha(n)-1}))$,
$2^{\alpha(n)-3}v_1^{3}(0,0)\neq0$.
\end{proof}

\appendix
\section{Comparison with previous TC-results}\label{comprev}
The following general lower bound for the topological complexity of lens 
spaces was proved by Farber and Grant in~\cite {FG}:

\begin{theorem}[{\cite[Theorem~11]{FG}}]\label{fargra}
Let $k$ and $\ell$ be integers with $0\leq k,\ell\leq m$. If $t$ does not divide
$\binom{k+\ell}{k}$, then $\TC(L^{2m+1}(t))\geq2(k+\ell)+1$.
\end{theorem}

The main task in this section is to complete the details 
of~(\ref{previous2FG}) and, therefore, of the fact that, 
for lens spaces of the form $L^{2n+3}(2^{\alpha(n)-1})$, 
Corollary~\ref{cotainferior} improves on Theorem~\ref{fargra} 
by arbitrarily large amounts. 

\begin{proposition}\label{Compara}
For a positive integer $n$ with $\alpha(n)\geq2$ and 
$\alpha(n+1)\geq\alpha(n)$, $\max\left\{k+\ell \,:\, 0\leq k,\ell\leq n+1
\,\,\,\mathrm {and}\,\,\, 2^{\alpha(n)-1} \nmid \binom{k+\ell}{k}
\right\}=2n-2^{\nu(n)+1}-1$.
\end{proposition}

\begin{nota}\label{understood}{\em
The relevance of the hypothesis $\alpha(n+1)\geq\alpha(n)$ comes 
from~\cite[Example~5.9]{symmotion}: the
topological complexity of $L^{2n+3}(2^{\alpha(n)-1})$ is 
completely understood for $\alpha(n +1) < 
\alpha(n)$---in which case Theorem~\ref{fargra} is 
optimal. On the other hand, note that the condition
$\alpha(n+1)\geq\alpha(n)$ is equivalent to $\nu(n+1)\leq 1$, which holds
except for $n\equiv3\bmod4$.
}\end{nota}

The proof of Proposition~\ref{Compara} for $\nu(n)=0$ reduces to a
simple checking:
using the formulas $\nu\binom{a}{b}=\alpha(a-b)+\alpha(b)-\alpha(a)$
and $\alpha(k-1)=\alpha(k)-1+\nu(k)$ it is easy to check that 
$\nu\binom{2n+2}{n+1}$, $\nu\binom{2n+1}{n+1}$, $\nu\binom{2n}{n+1}$,
$\nu\binom{2n}{n}$, $\nu\binom{2n-1}{n+1}$, $\nu\binom{2n-1}{n}$,
$\nu\binom{2n-2}{n+1}$, $\nu\binom{2n-2}{n}$, and $\nu\binom{2n-2}{n-1}$
are all greater than or equal to $\alpha(n)-1$, while 
\begin{eqnarray*}
\nu\binom{2n-3}{n} & \!\!\!= &\!\!\!\alpha(n)+ \alpha(n-3)- \alpha(2(n-1)-1)
\hspace{3.65cm}\\
& \!\!\!=& \!\!\!\alpha(n)+ \alpha(n-2)-1 +\nu(n-2)-\alpha(n-1)-\nu(n-1)\\
& \!\!\!=& \!\!\!\rule{0mm}{5.2mm}\alpha(n)-2+\nu(n-2)\;=\;\alpha(n)-2.
\end{eqnarray*}
However, the argument for $\nu(n)>0$ is arithmetically cumbersome and, for the reader's benefit, we split the required verifications into a few preliminary steps. Note that, for $\nu(n)>0$, the inequality $$\max\left\{k+\ell \,:\, 0\leq k,\ell\leq n+1 \,\,\,\mathrm {and}\,\,\, 2^{\alpha(n)-1} \nmid \binom{k+\ell}{k} \right\}\leq 2n-2^{\nu(n)+1}-1$$ in Proposition~\ref{Compara} is a consequence of the following result:

\begin{lema}\label{com-fg}
Let $k,n,$ and $t$ be non-negative integers with $\alpha(n)\geq2$. 
The inequality
$\nu\binom{t}{k}\geq \alpha(n)-1$ holds provided
\begin{eqnarray}
&&\!\!\!\!\!\!\!\!\!\!\!\!\!\!\!\!\!\!\!\!\!\!\!\!\nu(n)>0,\hspace{6.45cm}\label{bin1}\\
&&\!\!\!\!\!\!\!\!\!\!\!\!\!\!\!\!\!\!\!\!\!\!\!\!2n-2^{\nu(n)+1}\leq t\leq 2n+2, \mbox{\,\,\, and }\hspace{3.7cm}\label{bin2}\\
&&\!\!\!\!\!\!\!\!\!\!\!\!\!\!\!\!\!\!\!\!\!\!\!\!\left[\frac{t+1}{2}\right]\leq k\leq n+1,\hspace{5.85cm} \label{bin3}
\end{eqnarray} 
\end{lema}

In preparation for the proof of Lemma~\ref{com-fg}, we
introduce some supplementary notation which will be in force through the 
rest of the section. Set $\nu=\nu(n)$ and let
$n=2^{\nu_0}+2^{\nu_1}+\cdots$ with $\nu=\nu_0<\nu_1<\cdots$ be the
binary expansion of $n$. From~(\ref{bin2}) we have
$$2n-2^{\nu+1}=2^{\nu_1+1}+2^{\nu_2+1}+\cdots\leq t \leq 2+
2^{\nu_0+1}+2^{\nu_1+1}+\cdots \nonumber$$ and, if we set 
$t_0=t-(2^{\nu_1+1}+2^{\nu_2+1}+\cdots)$, we have
\begin{eqnarray}\label{m5}
0\leq t_0 \leq 2+ 2^{\nu_0+1}.
\end{eqnarray} Further, since $\nu_0\geq1$, it is
clear that, in the binary expansion of $t_0$, only powers
$2^{j}$ with $j\leq\nu_0+1$ are involved. In particular, from $t=t_0+
2^{\nu_1+1}+2^{\nu_2+1}+\cdots$ we have
\begin{eqnarray}\label{m6}
\alpha(t)=\alpha(t_0)+\alpha(2^{\nu_1+1}+2^{\nu_2+1}+\cdots)=
\alpha(t_0)+\alpha(n)-1.
\end{eqnarray} Likewise, if we set
$k_0=k-(2^{\nu_1}+2^{\nu_2}+\cdots)$, from (\ref{bin3}) one has
{\setlength\arraycolsep{2pt}
\begin{eqnarray}
\left[\frac{t_0+1}{2}\right]+2^{\nu_1}+2^{\nu_2}+\cdots &= &
\left[\frac{t_0+1+2^{\nu_1+1}+{}\cdots}{2}\right] \nonumber \\
& = & \left[\frac{t+1}{2}\right]\leq
k\leq 1+2^{\nu_0}+2^{\nu_1}+ {}\cdots,\nonumber
\end{eqnarray}} that is
\begin{eqnarray}\label{m7}
\left[\frac{t_0+1}{2}\right]\leq k_0\leq 1+2^{\nu_0}.
\end{eqnarray} 
Recalling again the assumption
$\nu_0\geq1$, it is now clear that, in the binary expansion of $k_0$ only \
powers $2^{j}$ with $j\leq\nu_0$ are involved. In particular, from $k=k_0+
2^{\nu_1}+2^{\nu_2}+\cdots$ we obtain
\begin{eqnarray}\label{m8}
\alpha(k)=\alpha(k_0)+\alpha(2^{\nu_1}+2^{\nu_2}+\cdots)=
\alpha(k_0)+\alpha(n)-1.
\end{eqnarray} 
Lastly, we claim that
\begin{equation}\label{aux1}
t_0-k_0<2^{\nu_0+1}.
\end{equation}
For, otherwise, from (\ref{m5}) and (\ref{m7}) we would have
\begin{eqnarray}\label{m9}
2+2^{\nu_0+1}\geq t_0\geq
k_0+2^{\nu_0+1}\geq\left[\frac{t_0+1}{2}\right]+2^{\nu_0+1},
\end{eqnarray} 
so that $2\geq\left[\frac{t_0+1}{2}\right]$, i.e.~$t_0\leq4$.
But~(\ref{m9}) also gives $$t_0\geq k_0+2^{\nu_0+1}\geq
2^{\nu_0+1}\geq2^2=4,$$ implying in fact $t_0=4$. However,
the inequality $t_0\geq \left[\frac{t_0+1}
{2}\right]+2^{\nu_0+1}$ in~(\ref{m9}) would now say $4\geq2+2^{\nu_0+1}\geq2+4=6$,
which is a contradiction.

\begin{proof}[Proof of Lemma~\emph{\ref{com-fg}}] 
From (\ref{m6})~and~(\ref{m8}) we obtain
\begin{eqnarray*}
\nu\mbox{$\binom{t}{k}$} & = &\alpha(k)+ \alpha(t-k)- \alpha(t)
\hspace{5.55cm}\\
& =& \alpha(k_0)+ \alpha(n)-1 + \alpha(t-k)-\alpha(t_0)-\alpha(n)+1\\
&= & \alpha(k_0)+ \alpha(t-k)-\alpha(t_0).
\end{eqnarray*}
Hence, taking into account the definitions of $k_0$ and $t_0$,
we have
\begin{eqnarray}\label{m10a}
\nu\mbox{$\binom{t}{k}$} & = & \alpha(k_0)+ \alpha(t_0-k_0+2^{\nu_1}+
2^{\nu_2} +\cdots)-\alpha(t_0).\hspace{0.35cm}
\end{eqnarray} 
At this point we consider two possibilities.

\medskip\noindent{\bf Case $k_0\leq t_0$}. $t_0-k_0$ is a non negative
integer whose binary expansion involves exclusively powers $2^{j}$ with
$j\leq\nu_0$, by virtue of~(\ref{aux1}). Thus~(\ref{m10a}) transforms into
\begin{eqnarray}
\nu\mbox{$\binom{t}{k}$} & = & \alpha(k_0)+
\alpha(t_0-k_0)+\alpha(2^{\nu_1}+ 2^{\nu_2} +\cdots)-\alpha(t_0)
\hspace{1.45cm}\nonumber\\
& =& \alpha(k_0)+ \alpha(t_0-k_0)-\alpha(t_0)+ \alpha(n)-1 \nonumber\\
&= & \nu\mbox{$\binom{t_0}{k_0}$}+ \alpha(n)-1\geq \alpha(n)-1,\nonumber
\end{eqnarray} as asserted by Lemma~\ref{com-fg}.

\medskip\noindent {\bf Case $k_0> t_0$}. $k_0 - t_0$ is a
positive integer whose binary expansion involves exclusively powers $2^{j}$
with $j\leq\nu_0$, by virtue of~(\ref{m7}). Thus, the central term
on the right hand side of (\ref{m10a}) becomes

\vspace{-6mm}
\begin{eqnarray*}
\alpha(t_0-k_0+ 2^{\nu_1}+2^{\nu_2} \cdots)& \!\!=\!\!
& \alpha(2^{\nu_2}+\cdots)+
\alpha(2^{\nu_1}-(k_0-t_0))\\
&\!\!=\!\! & \alpha(n)-2+ \alpha(2^{\nu_1}-1-(k_0-t_0-1))\\
&\!\!=\!\! & \alpha(n)-2+ \nu_1-\alpha(k_0-t_0-1)\\
&\!\!=\!\! & \alpha(n)-2+ \nu_1-(\alpha(k_0-t_0)-1+\nu(k_0-t_0)),
\end{eqnarray*} 

\vspace{-4mm}\noindent
so that
\begin{eqnarray*}\label{m10}
\nu\binom{t}{k} & = & \alpha(k_0)- \alpha(t_0)+\alpha(n)-2+\nu_1-
\alpha(k_0-t_0)+1-\nu(k_0-t_0)\\
&=& \alpha(k_0)- \alpha(t_0)-
\alpha(k_0-t_0)+\alpha(n)-1+\nu_1-\nu(k_0-t_0)\\
&=& -\nu\binom{k_0}{t_0}+\alpha(n)-1+\nu_1-\nu(k_0-t_0).
\end{eqnarray*}
Thus, under the present hypothesis (summarized as
$0\leq t_0 < k_0 \leq 1+ 2^{\nu_0}$ with $0<\nu_0<\nu_1$),
the conclusion of Lemma~\ref{com-fg} is equivalent to
$\nu_1\geq \nu(k_0-t_0)+ \nu\binom{k_0}{t_0}$. The proof is then 
complete in view of Lemma~\ref{aux2} below.
\end{proof}

\begin{lema}\label{aux2}
For integers $i,j,$ and $\mu$ with $\mu>0$ and $0\leq j<i\leq 2^{\mu+1}-1$, 
we have $\mu+1\geq \nu(i-j)+\nu\binom{i}{j}$.
\end{lema}

\begin{proof} The hypothesis means that the binary expansions of $i,j,$ and
$i-j$ involve exclusively powers $2^\ell$ with $\ell\leq\mu$, say
$$
\begin{matrix}
i-j \;=\; 2^{i_0}+2^{i_1}+\cdots+2^{i_a},\\
\;\;\;\;\;\,j \;=\; 2^{j_0}+2^{j_1}+\cdots+2^{j_b},
\end{matrix}
$$
with $i_0<i_1<\cdots<i_a\leq\mu$ and $j_0<j_1<\cdots<j_b\leq\mu$, where
$i_0 =\nu (i-j) $. By Kummer's
theorem~\cite{Kummer}, the number $C$ of binary carries in the sum of 
$i-j$ and $j\,$ is $C=\nu\binom{i}{j}$. But the condition
$i_a\leq\mu$ implies that the maximum possible of such carries is
$\mu-(i_0-1)$. Thus, $\nu\binom{i}{j}=C\leq\mu-i_0+1,$ as asserted.
\end{proof}

\begin{proof}[Proof of Proposition~\emph{\ref{Compara}}] It suffices to check that
$\nu\binom{2n-2^{\nu(n)+1}-1}{n}<\alpha(n)-1$ for $\nu(n) > 0$. Setting as 
above $\nu=\nu(n)$, we have
\begin{eqnarray*}
\nu\mbox{$\binom{2n-2^{\nu+1}-1}{n}$} & = &\alpha(n)+
\alpha(n-2^{\nu+1}-1)- \alpha(2(n-2^{\nu})-1)\hspace{5mm}\\
& =& \alpha(n)+ \alpha(n-2^{\nu+1})-1+\nu(n-2^{\nu+1})\\
& & {}-\left(\alpha(n-2^{\nu})-1+\nu(2(n-2^{\nu}))
\rule{0mm}{4mm}\right)\\
& =& \alpha(n)+ \alpha(n-2^{\nu+1})-1+\nu(n-2^{\nu+1})\\
& & {}- \,\alpha(n-2^{\nu})-\nu(n-2^{\nu}).
\end{eqnarray*} 
Since $n=2^{\nu}a$, where $a$ is an odd
integer greater than 1 (recall that $\alpha(n)\geq2$), we now have
\begin{eqnarray*}
\nu\mbox{$\binom{2n-2^{\nu+1}-1}{n}$} & \!=\! &\alpha(2^{\nu}\cdot a)+
\alpha(2^{\nu}(a-2))-1+\nu(2^{\nu}(a-2))\\
& & {}-\alpha(2^{\nu}(a-1))-\nu(2^{\nu}(a-1))\\
& \!=\!& \alpha(a)+\alpha(a-2)-1+ \nu -\alpha(a-1)-\nu-
\nu(a-1)\\
& \!=\!&  \alpha(a)+\alpha(a-1)-1+\nu(a-1)-1 -\alpha(a-1)-\nu(a-1)\\
& \!=\!& \alpha(a)-2= \alpha(n)-2< \alpha(n)-1,
\end{eqnarray*}
as asserted.\end{proof}

We close the paper with indications on how to solve a gap in~\cite{G}---so to justify our use of~(\ref{valCT}). In Lemma~4.1 of that paper it is claimed that there is no $\mathbb{Z}_t$-axial map
\begin{equation}\label{axialinexistente}
\beta\colon L^{2m+1}(t)\times L^{2m+1}(t)\to L^{2m+1}(t)
\end{equation}
when $t>2$---that is, a map $\beta$ for which the diagram in Proposition~\ref{biequi} (with $\ell=m$, and where $2^e$ is replaced by $t$) is homotopy commutative. Such an assertion (which is well known to fail for $t=1$) is not well argued for a general $t$ in the proof of~\cite[Lemma~4.1]{G}: although the argument correctly shows that a $\mathbb{Z}_t$-axial map~(\ref{axialinexistente}) can exist only when both $t$ and $m+1$ are prime powers\footnote{Although strictly a different property, the converse of this assertion would seem to be related to the stable parallelizability of lens spaces (a well-understood property by~\cite{litang}), and to the existence of $H$-space multiplications on localized lens spaces~(a well-understood property by~\cite{browder,harper}).}, the incompatibility with the relation ``$x^2=0$'' asserted at the end of that argument (and which really meant to use cohomology with $\mathbb{Z}_t$-coefficients) overlooked the graded commutativity in the cohomology ring $$H^*(L^{2m+1}(t)\times L^{2m+1}(t);\mathbb{Z}_t)\cong H^*(L^{2m+1}(t);\mathbb{Z}_t)\otimes H^*(L^{2m+1}(t);\mathbb{Z}_t).$$ In any case, what is important for our purposes is that the above problem hurts no other result in~\cite{G}. Indeed, Lemma~4.1 in that paper is used only in the proof of its main Theorem~2.9---stated here as the estimates for $\varepsilon_{n,t}$ in~(\ref{valCT}). Now, the easy proof of the inequality $\varepsilon_{n,t}\geq0$ is given in~\cite[Lemma~3.1 and Proposition~3.2]{G} independently of the problematic Lemma~4.1. It is the inequality $\varepsilon_{n,t}\leq1$, asserted for even $t$, the one that makes use of the potentially faulty Lemma~4.1. However, as noted in~\cite[Proposition~24]{GV}, the argument in the proof of~\cite[Lemma~17]{GV} shows that~\cite[Lemma 4.1]{G} is true even when $t$ is a 2-power.

\bigskip

{\sc Departamento de Matem\'aticas

Centro de Investigaci\'on y de Estudios Avanzados del IPN

M\'exico City 07000}

{\it E-mail address:} {\bf jesus@math.cinvestav.mx}
 
\bigskip\medskip
 
{\sc Colegio de Ciencia y Tecnolog\'{\i}a

Universidad Aut\'onoma de la Ciudad de M\'exico

Mexico City 09790}

{\it E-mail address:} {\bf maurilio.velasco.fuentes@uacm.edu.mx}
 
\bigskip\medskip
{\sc Mathematics Department

Johns Hopkins University

Baltimore, Maryland 21218}
 
{\it E-mail address}: {\bf wsw@math.jhu.edu}


\begin{thebibliography}{99}  

\bibitem{AGJ}{J.~Adem, S.~Gitler, I.~M.~James: {\it On axial maps of a 
certain type}, Bol.~Soc.~Mat. Mexicana {\bf 17} (1972) 59--62.}

\bibitem{browder}{W.~Browder: {\it Higher torsion in $H$-spaces},
Trans.~Amer.~Math.~Soc.{\bf 108} (1963) 353--375.} 

\bibitem{DD}{D.~M.~Davis: {\it A strong nonimmersion theorem for real 
projective spaces}, Ann.~of Math.~{\bf 120} (1984) 517--528.}

\bibitem{F}{M.~Farber: {\it Topological complexity of motion planning},
Discrete Comput. Geom.~{\bf 29} (2003) 211--221.}

\bibitem{FG}{M.~Farber; M.~Grant: {\it Robot motion planning, weights 
of cohomology classes, and cohomology operations}, 
Proc.~Amer.~Math.~Soc.~{\bf 136} (2008) 3339--3349.}

\bibitem{FTY}{M.~Farber; S.~Tabachnikov; S.~Yuzvinsky: {\it 
Topological robotics: motion planning in projective spaces},
Int.~Math.~Res.~Not.~{\bf 34} (2003) 1853--1870.}

\bibitem{GV}{J.~M.~Garc\'{\i}a-Calcines and L.~Vandembroucq: {\it
On the topological complexity and the homotopy cofibre of the diagonal map}, 
preprint arXiv:1106.4943v1 [math.AT].}

\bibitem{GN}{J. Gonz\'alez: {\it Connective K-theoretic Euler classes 
and nonimmersions of $2^k$ lens spaces}, J.~London Math.~Soc.~(2),
{\bf 63} (2001) 247--256.}

\bibitem{cfimm}{J.~Gonz\'alez: {\it A generalized Conner-Floyd
conjecture and the immersion problem for low $2$-torsion lens
spaces}, Topology {\bf 42} No. 4 (2003) 907--927.}

\bibitem{G}{J.~Gonz\'alez: {\it Topological robotics in lens spaces},
Math.~Proc.~Cambridge Phil.~Soc. {\bf 139} (2005) 469--485.}

\bibitem{symmotion}{J.~Gonz\'alez and P.~Landweber: {\it Symmetric topological 
complexity of projective and lens spaces}, Algebr.~Geom.~Topol.~{\bf 9}
(2009) 473–494.}

\bibitem{GZ}{J.~Gonz\'alez and L.~Z\'arate: {\it BP-theoretic instabilities 
to the motion planning problem in 4-torsion lens spaces}, Osaka 
J.~Math.~{\bf 43} (2006) 581–596.}

\bibitem{harper}{J.~Harper: {\it Regularity of finite $H$-spaces},
Illinois J.~Math.~{\bf 23} (1979) 330--333.}

\bibitem{Kummer}{E.~E.~Kummer: {\it \"Uber die erganzungss\"atze zu den 
allgemernen reciprocit\"atsgsetzen}, J.~Reigne Angew. Math., {\bf 44} 
(1852) 93--146.}

\bibitem{litang}{B.~H.~Li and Z.~Z.~Tang: {\it Codimension $1$ and $2$ 
immersions of lens spaces}, in \emph{Differential geometry and topology} 
(Tianjin, 1986--87), Lecture Notes in Math.~{\bf 1369} (1989) 152--163.}

\bibitem{Nakos}{G.~Nakos: On the Brown-Peterson homology of certain classifying spaces, Ph.D. Thesis, The Johns Hopkins University, 1985.}

\bibitem{Velasco}{M.~Velasco: On the $ku_*$-homology of $\mathbb{Z}_{2^k}\times\mathbb{Z}_{2^e}$ and its application to the topological complexity of lens spaces, Ph.D. Thesis, CINVESTAV-IPN, 2010.}

\end{thebibliography}
\end{document}